%% file: mf.tex
\documentclass[review,onefignum,onetabnum]{siamart190516}
\allowdisplaybreaks

\input{ex_shared}
\usepackage{comment}

\newcommand{\revision}[1]{#1}

\begin{document}

\nolinenumbers

\maketitle

\begin{abstract}
This paper studies a general class of stochastic population processes in which agents interact with one another over a network. Agents update their behaviors in a random and decentralized manner according to a policy that depends only on the agent's current state and an estimate of the macroscopic population state, given by a weighted average of the neighboring states. When the number of agents is large and the network is a complete graph (has all-to-all information access), the macroscopic behavior of the population can be well-approximated by a set of deterministic differential equations called a {\it mean-field approximation}. For incomplete networks such characterizations remained previously unclear, i.e., in general whether a suitable mean-field approximation exists for the macroscopic behavior of the population. The paper addresses this gap by establishing a generic theory describing when various mean-field approximations are accurate for \emph{arbitrary} interaction structures. 

Our results are threefold. Letting $W$ be the matrix describing agent interactions, we first show that a simple mean-field approximation that incorrectly assumes a homogeneous interaction structure is accurate provided $W$ has a large spectral gap. Second, we show that a more complex mean-field approximation which takes into account agent interactions is accurate as long as the Frobenius norm of $W$ is small. Finally, we compare the predictions of the two mean-field approximations through simulations, highlighting cases where using mean-field approximations that assume a homogeneous interaction structure can lead to inaccurate qualitative and quantitative predictions. 
\end{abstract}



\section{Introduction}
We consider the problem of analysis and approximation of stochastic dynamics that emerge in large networks of interacting agents. 
It is often of interest to study the {\it macroscopic} behavior of such processes -- that is, the evolution of the fraction of agents playing a particular action -- which we refer to as the {\it population process}. The analysis of population processes are important in several disciplines, including evolutionary biology \cite{jonker,JMS, schuster_replicator,hofbauer_evolutionary_game_dynamics}, epidemiology \cite{hethcote, fall_lyapunov, mei_network_epidemics, khanafer2016, mieghem2009, pare_epidemics_review}, game theory and controls  \cite{szabo2007, madeo_network_pop_games, como_network_pop_games, ohtsuki2006, iacobelli2016lumping, nce, decentralized_mfg} and economics \cite{sandholm2010popgames, weibull}. However, when the number of agents involved is large, such processes are challenging to study analytically and computationally. A common workaround is to approximate complex stochastic processes with a differential equation known as a \emph{mean-field approximation} (MFA), which is generally more amenable to analysis. 
Perhaps most strikingly, MFAs provide a tractable way to study the {\it emergent behavior} of large-scale systems, and several authors have used them as a starting point for the tractable multi-agent control of large populations (see, e.g., \cite{Lasry2007}). The success of these controls in practical systems relies fundamentally on the {\it accuracy} of MFAs.

The accuracy of MFAs was first studied formally by Kurtz \cite{kurtz1976}, who showed that population processes concentrate around their MFAs over finite time horizons, provided that agent interactions are {\it homogeneous} (i.e., all agents act based on the distribution of actions in the {\it full} population). However, in applications such as network games and epidemiology, it is more typical that agent interactions are {\it heterogeneous} (i.e., agents act based on the distribution of actions within a {\it subset} of the population). Two natural questions therefore emerge: (1) Are mean-field approximations still useful when interactions are heterogeneous? (2) Are there alternate approximations one can use when the mean-field approximation is inaccurate? 
Despite the fundamental importance of these questions, prior work only provides partial answers. Towards answering (1), existing work (see, e.g., \cite{OliveraReis_diverging_degree, DelattreGiacominLucon2016, CoppiniDietertGiacomin2020, SantosThermo}) has shown that MFAs are accurate under simple models of {\it random} interactions (e.g., given by an Erd\H{o}s-R\'{e}nyi graph). Towards answering (2), many authors have proposed more complex MFAs depending on the interaction structure which are known to be accurate when the interaction structure converges to a well-defined limiting object \cite{spatial_sis, hwang_spatial, Kar_supernode, KELIGER2022, bayraktar2021graphon, graphon_mfg_1, santos_moura_multipartite, ganguly_ramanan}. 
While these works make significant progress in answering (2), they provide little insight into population processes driven by {\it finite and arbitrary} interactions. Moreover, in all the aforementioned works related to (2), it is unclear whether it is truly {\it necessary} to use complex MFAs instead of the simpler MFAs which assume homogeneous interactions. Indeed, empirical results indicate that in many instances of heterogeneous interactions, the simpler MFAs may still provide useful predictions \cite{accuracy_MFT}. 
\vspace{-0.2cm}
\subsection{Contributions} 
In this work, we establish a new, generic theory that describes \emph{when} and \emph{how} insights from MFAs can be translated to their stochastic counterparts under minimal assumptions. 
In particular, we provide rigorous answers to questions (1) and (2) for {\it generic} interaction structures. Our model of agent interactions is a multi-agent Markov jump process, wherein agents update their actions at a rate which depends on the empirical distribution of their (appropriately-defined) neighborhood; such models are natural in game theory \cite{szabo2007, madeo_network_pop_games, como_network_pop_games, ohtsuki2006, iacobelli2016lumping, sandholm2010popgames}, epidemiology \cite{fall_lyapunov, mei_network_epidemics, khanafer2016, mieghem2009, pare_epidemics_review, van_mieghem_intertwined, van_mieghem_2014} and interacting particle systems \cite{gleeson_high_accuracy, gleeson_approx}. We relate the deviation between the stochastic population process and the corresponding MFA to the \emph{density} of agent interactions, showing in particular that denser interactions lead to tighter approximations by MFAs. To make this precise, we introduce two measures of density: the first, which we call the \emph{local density}, measures the strength of local or pairwise interactions between agents, and the second, called the \emph{spectral density}, depends on the global structure of interactions as measured by operator-theoretic properties of the matrix capturing the interactions. Our results and methods are succinctly described below.

{\it Robustness of classical MFAs.} Our first main result shows that if the interactions between agents are {\it spectrally dense}, then a simple MFA which incorrectly assumes a homogeneous interaction structure between agents (henceforth called the {\it classical MFA}), is, surprisingly, a good approximation for stochastic population processes driven by heterogeneous interactions. In other words, the usage of more complex MFAs to approximate the stochastic behavior is unnecessary such cases. Our proof uses the crucial fact that our definition of spectral density bounds the average deviation between the empirical distribution of neighboring actions and that of the full population. 
As a result, the population process resembles a system driven by {\it homogeneous} interactions, for which the classical MFA is known to be accurate.
For details, see Theorem \ref{thm:deviation_rapidly_mixing}.

{\it General approximation of population processes by MFAs.} We show that if agent interactions are {\it locally dense}, then a MFA which takes the interactions into account, commonly known as a $N$-intertwined mean-field approximation (NIMFA), provides a good approximation of the stochastic population process. This is our most general and technically challenging result. 
In prior work on homogeneous interactions, the population-level behavior is a Markov process, and standard techniques for Markov processes could be applied to prove the accuracy of MFAs. 
However, since we allow for {\it arbitrary} interactions between agents in this work, the population-level behavior is {\it not} Markov due to the lack of symmetry in the system.
To get around this issue, we construct a family of auxiliary processes which {\it can} be studied in an autonomous manner. At the same time, our design of the auxiliary processes ensures that we do not incur a significant loss in our ultimate probabilistic bounds. Indeed, we are able to show that the probability of error between the population process and the NIMFA decays exponentially in $N$ when interactions are locally dense, which is also known to be the case for homogeneous interactions \cite{BenaimWeibull, sandholm_staudigl_2018}. For details, see Theorem \ref{thm:deviation}.

{\it Choosing the right approximation.} Although our results show that the NIMFA is generally more accurate than the classical MFA, it is still unclear whether the classical MFA is still useful in cases where agent interactions are {\it not} locally dense. Through simulations, we show that the answer depends on the {\it initial conditions} of the stochastic process. When agents' initial states are {\it random}, the classical MFA may still yield accurate predictions. On the other hand, we show that under more structured initial conditions, the classical MFA leads to highly inaccurate predictions {in both transient and steady-state regimes}.
For details, see Section \ref{sec:choosing}.
\vspace{-0.2cm}

\subsection{Further related work} While the general concept of a mean-field approximation is of broad utility, the form it takes (e.g., ordinary differential equation, partial differential equation, stochastic differential equation) can depend drastically on the specific model of multi-agent interactions that are considered. In this work, we focus our attention to a natural class of Markov jump processes, though we note that other models (e.g., interacting diffusions, decentralized optimal control) have also received significant attention (see, e.g., \cite{OliveraReis_diverging_degree, DelattreGiacominLucon2016, CoppiniDietertGiacomin2020, ganguly_ramanan, bayraktar2021graphon, graphon_mfg_1, nce, decentralized_mfg}). \\
\indent Most of the literature related to our model studies the convergence of the population process to a mean-field approximation in asymptotic regimes where the number of agents tend to infinity \cite{hwang_spatial, Kar_supernode, santos_moura_multipartite, KELIGER2022,ganguly_ramanan}, however such methods do not apply to our case since we consider the case of {\it finite} and {\it arbitrary} interactions for which such limiting behaviors are not well-defined. In the case of {\it homogeneous} interactions, Bena\"{i}m and Weibull \cite{BenaimWeibull} provided a sharp \emph{non-asymptotic} analysis of the stochastic population process. Since we consider arbitrary and fixed interaction structures, Bena\"{i}m and Weibull's non-asymptotic techniques are a natural inspiration for the ones we develop in this work. Finally, we remark that after an initial draft of our work was posted online \cite{sridhar_kar_arxiv_2021}, we came across related work by Horv\'{a}th and Keliger \cite{horvath2022accuracy}. They also consider the accuracy of the NIMFA for a slightly more general model of agent interaction under finite and arbitrary interaction structures. However, their probability bounds are considerably weaker than ours (see Remark \ref{remark:comparison} for a more detailed comparison). 
\vspace{-0.2cm}
\subsection{Organization} The rest of the paper is structured as follows. Section \ref{sec:notation} introduces some notation that we use throughout the paper. Section \ref{sec:model} formally describes the class of stochastic processes we consider, derives the corresponding MFAs and discusses related work. Section \ref{sec:applications} highlights concrete applications of our theory to game theory and epidemiology. In Section \ref{sec:results}, we describe our main results -- specifically, our first and second contributions. In Section \ref{sec:choosing}, we investigate through simulations the accuracy of the MFAs we consider, addressing our third main contribution. The remaining sections contain the proofs of our main results. 
\vspace{-0.1cm}
\section{Notation}
\label{sec:notation}
Let $\reals$, $\reals_{\ge 0}, \integers, \integers_{\ge 0}$ denote the set of real numbers, non-negative real numbers, integers, and non-negative integers, respectively. For a finite set $A$, $\reals^A$ is the set of $|A|$-dimensional vectors with entries indexed by elements of $A$. For $1 \le p < \infty$ and $x \in \reals^n$, we define the $\ell_p$ norm $\| x \|_p : = ( \sum_{i = 1}^n |x_i|^p )^{1/p}$. We also define $\| x \|_\infty : = \max_{1 \le i \le n} |x_i |$. For a finite set $A$, we define the set of \emph{mixed states} to be $\Delta(A) : = \{ x \in \reals_{\ge 0}^A : \sum_{a \in A} x_a = 1 \}$, and
{$\cV(A)$ denotes the set of standard basis vectors in $\mathbb{R}^A$.}
We also denote $\mathbf{0}$ and $\mathbf{1}$ to be vectors with entries that are all equal to 0 or 1, respectively. Throughout, we utilize standard asymptotic notation (e.g., $O(\cdot)$ and $\Omega( \cdot)$). Finally, we write $\mathbf{1}(\cdot)$ to be an indicator function that takes the value 1 when the condition inside the parentheses are met, else it is zero. 

\section{Stochastic population processes and MFAs}
\label{sec:model}

Consider a population of $N$ agents, indexed by the set $[N] : = \{1, \ldots, N \}$. At any point in time $t \ge 0$, each agent has an associated state $s_i(t)$ which is an element of a fixed, finite set $\cS$. The agent is also equipped with an independent Poisson clock with rate $r_i$ (that is, a Poisson process with rate $r_i$) such that when the agent's clock rings, the agent is allowed to update their state. Upon a clock ring, agents update their state via a probabilistic policy that is specified by the collection of functions $\boldsymbol{\rho} : = \{\rho_i^{\alpha \beta} \}_{\alpha, \beta \in \cS, i \in [N]}$ as well as a row-stochastic information aggregation matrix $W = [ w_{ij} ]_{i,j \in [N]} \in \reals^{N \times N}$. If agent $i$'s clock rings at time $t$, they first compute a {\it local population estimate} $\overline{Y}_i(t) \revision{= \{ \overline{Y}_i^\alpha(t) \}_{\alpha \in \cS}} \in \reals^{\cS}$. 
Specifically, if we define $Y_i^\alpha(t) : = \mathbf{1}( s_i(t) = \alpha)$, the local population estimates are computed according to
$$
\overline{Y}_i^\alpha(t) := \sum\limits_{j \in [N]} w_{ij} Y_j^\alpha(t).
$$
\revision{In other words, $\overline{Y}_i^\alpha(t)$ is the weighted fraction of agents in $i$'s neighborhood with state $\alpha$ at time $t$.}
Since $\overline{Y}_i(t)$ is a convex combination of neighboring agents' states, it can be interpreted as agent $i$'s estimate of the population behavior.\footnote{In this model, the local estimates are {\it linear} functions of the neighborhood states. In principle, one could use a nonlinear aggregation function, but we focus on the linear case as it is simple and natural to study. This form of local estimate formation through neighborhood averaging is relevant to and arises in problems of distributed information processing, see for example, DeGroot's consensus formation model~\cite{degroot1974reaching} as well as other models in optimization, control and game-theoretic computations in networks, see~\cite{tsitsiklis1986distributed,jadbabaie2003coordination,dimakis2010gossip,swenson2015empirical}.}
{We say $W$ is \emph{homogeneous} if $w_{ij} = 1/N$ for all $i,j$, and otherwise it is \emph{heterogeneous}.}
{In the homogeneous case}, the local population estimate reduces to the \emph{true population average} $Y_{av}(t)$, which is the case studied by Kurtz \cite{kurtz}. If $s_i(t) = \alpha$, agent $i$ changes their state to $\beta \in \cS$ with rate given by $\rho_i^{\alpha \beta} ( \overline{Y}_i(t))$. We assume the following about $\boldsymbol{\rho}$.

\begin{assumption}
\label{as:rho}
There exists a constant $L_\rho \ge 0$ such that for all $i \in [N]$ and all $\alpha, \beta \in \cS$, $\rho_i^{\alpha \beta} : \Delta(\cS) \to [0,1]$ and $\rho_i^{\alpha \beta}$ is $L_\rho$-Lipschitz.
\end{assumption}

The stochastic process we have described -- which we henceforth call a \emph{stochastic population process} -- can be succinctly represented as $\{ \mathbf{Y}(t) \}_{t \ge 0}$, where $\mathbf{Y}(t) = \{ Y_i(t) \}_{i \in [N]}$ is the set of agent states at time $t$. 
Formally, $\mathbf{Y}(t)$ is a continuous-time Markov process of jump type, and has state space {$\cV(\cS)^N$}.

For large populations, it is challenging to analytically study or even simulate stochastic population processes. To get around this issue, a typical approach is to instead study an appropriate mean-field approximation. The most basic such approximation assumes (potentially incorrectly) that agent behaviors are identical ($r_i = r$ and $\rho_i^{\alpha \beta} = \rho^{\alpha \beta}$ for all $i \in [N]$) and that the local population estimates are actually \emph{equal} to the true population average, given by $Y_{av}(t) : = \frac{1}{N} \sum_{i \in [N]} Y_i(t)$. The corresponding mean-field approximation $x(t)$ is then given by the following differential equation:
\begin{equation}
\label{eq:classical_mean_field}
\dot{x}^\alpha = r\sum\limits_{\beta \in \cS} \left(  x^\beta \rho^{\beta \alpha}(x) - x^\alpha \rho^{\alpha \beta }(x) \right) = :\phi^\alpha(x), \qquad \alpha \in \cS.
\end{equation}
Equation \eqref{eq:classical_mean_field}, which we call the \emph{classical mean-field approximation} (CMFA), can be explained as follows. Suppose that $x(t)$ tracks the true population average over time. The term $x^\beta \rho^{\beta \alpha}(x)$ represents the possibility that a single updating agent is in state $\beta$ and switches from $\beta$ to $\alpha$. Hence $\sum_{\beta \in \cS} x^\beta \rho^{\beta \alpha}(x)$ represents the \emph{influx} to the set of agents in state $\alpha$. On the other hand, $\sum_{\beta \in \cS} x^\alpha \rho^{\alpha \beta}(x)$ represents the possibility that the updating agent is currently in state $\alpha$ and decides to switch to another state; hence this term captures the \emph{outflux} from the set of agents in state $\alpha$. Significantly, the complex stochastic dynamics collapse into a $|\cS|$-dimensional ordinary differential equation (ODE), paving the way for a tractable analysis through both theory and simulation.

Even though the assumptions underlying the CMFA are overly simplistic, it often predicts the right qualitative behavior observed in practice, and in some cases, approximates the behavior of $Y_{av}(t)$ very well even when interactions are heterogeneous (see, e.g., \cite{accuracy_MFT}). In other words, the CMFA is quite robust. Various heuristic and semi-rigorous arguments have been proposed to explain this robustness \cite{accuracy_MFT,mieghem_approx, cong2012comparison}, but until this work there has been no formal theory to justify it.

To get around the restrictive assumptions of the CMFA, one can also consider a mean-field approximation that accounts for the various heterogeneities of the stochastic population process. In this MFA, the state variables are given by $\mathbf{y}(t) = \{ y_i(t) \}_{i \in [N]}$, where $y_i(t) \in \Delta( \cS)$. We also define
$$
\overline{y}_i(t) : = \sum\limits_{i \in [N]} w_{ij} y_j(t). 
$$
The MFA is given by
\begin{equation}
\label{eq:mean_field}
\dot{y}_i^\alpha = r_i \sum\limits_{\beta \in \cS} \left( y_i^\beta \rho_i^{\beta \alpha}( \overline{y}_i) - y_i^\alpha \rho_i^{\alpha \beta}( \overline{y}_i) \right) = : \Phi_i^\alpha ( \mathbf{y}), \qquad  i \in [N], \alpha \in \cS. 
\end{equation}
More compactly, we may also write $\dot{\mathbf{y}} = \Phi(\mathbf{y})$. Since \eqref{eq:mean_field} is a system of $N$ co-dependent differential equations, each of dimension $|\cS|$, we call \eqref{eq:mean_field} the \emph{$N$-intertwined mean-field approximation} (NIMFA). The NIMFA can be explained as follows. 
Conditioned on $\mathbf{Y}(t)$ and the event that agent $i$'s clock rings at time $t$,
the quantity $\rho_i^{\beta \alpha} ( \overline{Y}_i(t))$ is the probability that agent $i$ changes state from $\beta$ to $\alpha$. On the other hand $\sum_{\beta \in \cS} \rho_i^{\alpha \beta} ( \overline{Y}_i(t))$ is the probability that agent $i$'s state changes from $\alpha$ to another state (potentially including $\alpha$). Noting that the probability that agent $i$ updates in an interval of size $\delta$ is approximately $\delta r_i$ for small $\delta$, it follows that
$$
\E [ Y_i^\alpha(t + \delta) - Y_i^\alpha(t) \vert \mathbf{Y}(t)] = \delta r_i \sum\limits_{\beta \in \cS} \left( Y_i^\beta(t) \rho_i^{\beta \alpha} (\overline{Y}_i(t)) - Y_i^\alpha(t) \rho_i^{\alpha \beta} ( \overline{Y}_i(t)) \right) + o(\delta).
$$
Replacing $Y_j^\alpha(t)$ by $y_j^\alpha(t)$ for all $j \in [N], \alpha \in \cA$ and sending $\delta \to 0$ yields \eqref{eq:mean_field}. Consequently, $y_i(t)$ can be thought of as a first-order approximation of $\E [ Y_i(t) ]$. We can also approximate the population average $Y_{av}(t)$ by $y_{av}(t) : = \frac{1}{N} \sum_{i \in [N] } y_i(t)$. 

Compared to the CMFA, the NIMFA's history is more recent. It was initially studied in the context of epidemiology by Lajmanovich and Yorke \cite{lajmanovich1976}, who called it a $N$-group model. It later gained significant attention from the epidemiology, physics and controls community \cite{fall_lyapunov, mei_network_epidemics, khanafer2016, mieghem2009, pare_epidemics_review, van_mieghem_intertwined, van_mieghem_2014}, with the terminology ``NIMFA" introduced by Van Mieghem and co-authors \cite{van_mieghem_intertwined, van_mieghem_2014, mieghem2009}. NIMFAs have also been studied in the context of binary state dynamics (e.g., spin dynamics, voter models) \cite{gleeson_high_accuracy, gleeson_approx} as well as population games \cite{szabo2007, madeo_network_pop_games, como_network_pop_games, ohtsuki2006, iacobelli2016lumping}. Although the NIMFA is still quite complex compared to the CMFA, some analysis is still possible, such as the characterization of equilibria (see \cite{lajmanovich1976} as well as \cite{fall_lyapunov} and references within). A significant issue with the NIMFA is that until this work, there were no existing guarantees on how well it approximates its stochastic counterpart in general.

We make a few further remarks about the NIMFA. As it has the ability to account for heterogeneous agent behavior (e.g., distinct $r_i$ and $\rho_i^{\alpha \beta}$), it is more widely applicable than the CMFA. Even when agent behaviors are identical, however, it is not hard to find examples where the CMFA and the NIMFA yield different predictions for the population average. To see why this is the case in a bit more detail, recall that the CMFA is accurate under the assumption that the local population estimates $\overline{Y}_i(t)$ are equal to the population average $Y_{av}(t)$. However, the CMFA fails to capture what happens when the local population estimates deviate significantly from the population average, whereas the NIMFA \emph{can} capture such effects due to its granularity. We elaborate on the accuracy of the CMFA and NIMFA in Section \ref{sec:choosing}.

\section{Applications}
\label{sec:applications}

We highlight the importance of understanding stochastic population processes and MFAs through applications in game theory and epidemiology. 

\subsection{Population games}
\label{subsec:population_games}

In population games \cite{sandholm2010popgames}, agents can play actions in a finite set $\cS$, and the utility associated to a particular action depends on the action played by the individual as well as the \emph{distribution} of actions within the population.\footnote{In classical game theory, payoffs are allowed to depend on the individual actions of other players rather than the aggregate measure we discuss here.} We mathematically describe the utility structure by a function $\cU : \Delta( \cS) \to \reals^{\cS}$, with $\cU(x) = \{ \cU^\alpha(x) \}_{\alpha \in \cS}$ for $x \in \Delta(\cS)$. In particular, $\cU^\alpha(x)$ represents the utility of playing $\alpha$ when the distribution of actions within the population is given by $x$. While it is possible to study the static game described by $\cU$, an interesting question is \emph{how} agents interact with each other when their behaviors are motivated by the utilities given by $\cU$. The stochastic model describing agent interactions assumes that $r_i = 1$ for all $i \in [N]$ and $\rho_i^{\alpha \beta} = \rho^{\alpha \beta}$ for all $i \in [N], \alpha, \beta \in \cS$. There are various natural choices of $\boldsymbol{\rho}$ depending on how agents respond to utilities \cite{sandholm2010popgames}. A question of significant interest is whether such natural agent behaviors lead to desirable game-theoretic outcomes such as Nash equilibria. Researchers typically answer this question by a direct analysis of the corresponding CMFA or NIMFA \cite{sandholm2010popgames, network_games}. A key gap in the literature addressed in this paper is how well the MFAs approximate the true stochastic population process.


\subsection{Epidemic models}
\label{subsec:epidemic}
In this section, we consider the Susceptible-Infected-Susceptible (SIS) process, though we remark that our analysis can be readily extended to similar models. The stochastic SIS process in heterogeneous populations is parametrized by an interaction matrix $A \in \reals_{\ge 0}^{N \times N}$, an infection rate $b > 0$ and a recovery rate $\revision{\delta} > 0$. At any point in time, agents are either susceptible or infected. Directed interactions between agent $i$ and agent $j$ occur on the rings of a Poisson clock with rate $b A_{ij}$. If agent $i$ is susceptible and agent $j$ is infected when the interaction occurs, agent $i$ becomes infected. Once infected, an agent recovers and becomes susceptible again after an $\mathrm{Exponential}(\revision{\delta})$ amount of time. 

The stochastic SIS process turns out to be a special case of the dynamics we consider, which is made evident by the following equivalent construction. Let $\mathbf{Y}(t) = \{ ( Y_i^S(t), Y_i^I(t)) \}_{i \in [N]}  \in  \{ (1,0), (0,1) \}^N$ be the state variable tracking the evolution of the epidemic, where $Y_i^S(t) = 1$ if agent $i$ is susceptible at time $t$, else $Y_i^I(t) = 1$ if agent $i$ is infected at time $t$. For each $i \in [N]$, set $d_i : = b \sum_{j \in [N]} A_{ij}$ and $r_i : =  d_i + \revision{\delta}$. It follows from the definition of the stochastic SIS process that agent $i$ potentially updates their state (susceptible or infected) upon the rings of a Poisson clock with rate $r_i$. Next, we define the matrix $W : = b D^{-1} A$ where $D$ is a diagonal matrix with $D_{ii} = d_i$ so that in particular, $W$ is row-stochastic. Given $W$, we define the local population estimate $\overline{Y}_i(t) = ( \overline{Y}_i^S(t), \overline{Y}_i^I(t))$ in the usual sense. If agent $i$ is susceptible when the clock rings, they become infected with probability 
\begin{equation}
\label{eq:SI_transition}
\frac{ b \sum_{j \in [N]} A_{ij} Y_j^I(t) }{d_i + \revision{\delta}} = \frac{ \frac{b}{d_i} \sum_{j \in [N]} A_{ij} Y_j^I(t) }{ 1 + \revision{\delta} / d_i} = \left( 1 + \frac{\revision{\delta}}{d_i} \right)^{-1} \overline{Y}_i^I = : \rho^{SI}_i \left( \overline{Y}_i(t) \right).
\end{equation}
On the other hand if agent $i$ is infected when the clock rings, they become susceptible with probability 
\begin{equation}
\label{eq:IS_transition}
\frac{\revision{\delta}}{d_i + \revision{\delta}} =: \rho_i^{IS}\left( \overline{Y}_i(t) \right).
\end{equation}
From \eqref{eq:SI_transition} and \eqref{eq:IS_transition}, it is clear that the stochastic SIS process is indeed a special case of the Markov jump processes we consider, where the transition rates are linear or constant functions of the local population estimates.\footnote{To see this more formally, it is straightforward to show that the generator of the process induced by the transition functions in \eqref{eq:SI_transition} and \eqref{eq:IS_transition} is the same as the generator for the SIS process derived in \cite{mieghem2009}.} Correspondingly, the deterministic process $\dot{\mathbf{y}}= \Phi(\mathbf{y})$ is given by
$\dot{y}_i^I = b y_i^S \sum_{j \in [N]} A_{ij} y_j^I - \revision{\delta} y_i^I$ for $i \in [N]$,
with $\dot{y}_i^S = - \dot{y}_i^I$. As discussed in Section \ref{sec:model}, the bulk of the literature of epidemic models with heterogeneous agent interactions directly studies the NIMFA. 
Little was previously known of how well the NIMFA approximates the true population process. In \cite{mieghem2009}, Van Mieghem, Omic, and Kooij derived an expression for the variance of local population estimates and provided heuristic explanations for the accuracy of the NIMFA in certain parameter regimes. Our work provides a significant generalization of these ideas that holds under minimal assumptions on the stochastic population process, with probability bounds that are much sharper than variance-based estimates.


\section{Results}
\label{sec:results}

\subsection{Preliminaries}

It is well-known that the CMFA \eqref{eq:classical_mean_field} is a good approximation for the population average $Y_{av}(t)$ under homogeneous interactions ($W = \mathbf{1 1}^\top  / N$) and identical agent behavior (i.e., $r_i= r, \rho_i^{\alpha \beta}= \rho^{\alpha \beta}$). Perhaps the most quantitative version of this fact was proved by Bena\"{i}m and Weibull \cite{BenaimWeibull}, which we state below. 

\begin{theorem}
\label{thm:bw}
Let $N$ be the size of the population and assume that $W = \mathbf{11}^\top / N$ and that for all $i \in [N]$ and all $\alpha, \beta \in \cS$, $r_i = 1$, and $\rho_i^{\alpha \beta} = \rho$. Let $x(t)$ solve $\dot{x} = \phi(x)$ with initial condition $x(0) = Y_{av}(0)$. For any time horizon $T > 0$, there is a constant $c = c(T)$ such that for all $\epsilon > 0$,
$$
\p \left( \sup\limits_{0 \le t \le T} \norm{ Y_{av}(t) - x(t) }_\infty > \epsilon  \right) \le 2 | \cS | e^{ - c N \epsilon^2}.
$$
\end{theorem}

In words, Theorem \ref{thm:bw} states that the population average of the stochastic population process concentrates around the mean-field approximation \eqref{eq:classical_mean_field} over bounded time intervals. Notice that Theorem \ref{thm:bw} is a \emph{non-asymptotic} concentration result, unlike the original work of Kurtz \cite{kurtz1976}. This aspect of the theorem is particularly important for us, since $N \to \infty$ asymptotics are well-defined in the homogeneous interactions case but not for the finite, arbitrary interaction structures we consider. For this reason, we adapt the general approach taken by Bena\"{i}m and Weibull as opposed to others. Unfortunately, the proof of Theorem \ref{thm:bw} relies heavily on the symmetries induced by homogeneous interactions and therefore cannot be directly extended to account for general $W$. Nevertheless, we show that if $W$ is sufficiently dense, the MFAs \eqref{eq:classical_mean_field} or \eqref{eq:mean_field} are good approximations for the stochastic population process in the sense of Theorem \ref{thm:bw}. To make this result precise, we introduce two notions of density. For both, a small value indicates that $W$ is dense and a larger value indicates that it is sparse. Our first density measure depends on the strength of local interactions. 

\begin{definition}
The {\bf local density} of a row-stochastic matrix $W$ is 
$$
\theta(W) : = \sqrt{ \frac{1}{N} \sum\limits_{i,j \in [N]} w_{ij}^2 } = \frac{1}{\sqrt{N}} \| W \|_F.
$$
Moreover, we informally use the terminology ``locally dense'' to mean that $\theta(W)$ is a function of $N$ that tends to zero as $N \to \infty$. 
\end{definition}

\begin{remark}[Random walk matrices]
\label{remark:random_walk}
A special case of interest is when $W$ is a \emph{random walk matrix}; that is, if $G$ is an undirected graph on $[N]$ and $\mathrm{deg}(i)$ is the number of neighbors of $i$ in $G$, then $w_{ij} = 1/ \mathrm{deg}(i)$ if $(i,j)$ is an edge in $G$ and $w_{ij} = 0$ otherwise. In this special case, $\theta(W) = \sqrt{ \frac{1}{N} \sum_{i \in [N] } \frac{1}{\mathrm{deg}(i)} }$. 
In particular, $\theta(W)$ is small if all but a small fraction of nodes have high degree.
\end{remark}

Our next density measure depends on operator-theoretic properties of $W$. 

\begin{definition}
The {\bf spectral density} of a row-stochastic matrix $W$ is
$$
\lambda(W) : = \sup\limits_{x \in \reals^N \setminus \{0 \} : \mathbf{1}^\top x = 0 } \frac{ \| W x \|_2 }{ \| x \|_2}.
$$
Moreover, we informally use the terminology ``spectrally dense'' to mean that $\lambda(W)$ is a function of $N$ that tends to zero as $N \to \infty$.
\end{definition}

The terminology \emph{spectral density} comes from the observation that $\lambda(W)$ can be viewed as the spectral norm of $W$ restricted to the subspace of vectors $x \in \reals^N \setminus \{0 \}$ satisfying $\mathbf{1}^\top x = 0$. 

\begin{remark}[Complete graphs with link failures]
\label{remark:link_failures}
Let $G$ be the complete graph on $[N]$, and suppose that for each vertex $i$, there is a set of neighbors $F_i \subset [N]$ such that the edge $(i,j)$ is removed from $G$ for all $j \in N_i$ to form the graph $G'$ (this process models the failure of inter-agent links, for instance). To describe this more formally, let us define the matrices $W = \{ w_{ij} \}_{ij}$ and $H = \{ h_{ij} \}_{i,j}$, with entries given by 
\[
w_{ij} : = \begin{cases}
\frac{1}{N - 1 - | F_i|} & j \notin F_i \\
0 & j \in F_i \cup \{i \},
\end{cases} 
\hspace{0.5cm} 
\text{and} 
\hspace{0.5cm}
h_{ij}  := \begin{cases}
\frac{ |F_i| +1 }{N(N - 1 - |F_i|)} & j \notin F_i \\
-\frac{1}{N} & j \in F_i \cup \{i \}.
\end{cases}
\]
Notice that $W$ is the random walk matrix corresponding to $G'$, and can be written as $W = \mathbf{11}^\top / N + H$. Since $\lambda( \mathbf{11}^\top / N) = 0$, by the sub-additivity of $\lambda$ we can bound $\lambda(W)$ by the operator 2-norm of $H$, which is at most $\| H \|_F$. Through straightforward calculations, 
{if $\max_{i \in [N]} |F_i| = o(N)$ then}
$\| H \|_F^2 \le 2 ( \max_{i \in [N]} |F_i | + 1) / N = o(1)$.
{Hence $W$ is spectrally dense in this case.}
\end{remark}

\begin{remark}[Expander graphs]
\label{remark:expander}
Let $G = (V, E)$ be an undirected graph with vertex set $V = [N]$ and let $A$ be the (symmetric) adjacency matrix of $G$ with real eigenvalues $\mu_1, \ldots, \mu_N$ ordered from largest to smallest with respect to magnitude. Assuming that $G$ is $k$-regular, we have $\mu_1 = k$. Furthermore, if $W$ is the random walk matrix for $G$, then $W = A / k$ and it can be seen that $\lambda(W) = |\mu_2| / k$. If $|\mu_2|$ is much smaller than $k$, then $\lambda(W) \approx 0$, and $W$ can be thought of as a \emph{spectral perturbation} of the complete graph in light of Remark \ref{remark:spectral_density_homogenous}. Graphs for which $|\mu_2|$ is much smaller than $k$ are known as \emph{expanders}. Various families of random graphs, such as random regular graphs or Erd\H{o}s-R\'{e}nyi graphs, are known to be expanders \cite{erdos_renyi_eigenvalues1,regular_eigenvalues2,regular_eigenvalues1,friedman08}.
\end{remark}%
\begin{remark}[Extremal values of density measures]
\label{remark:spectral_density_homogenous}
When $W = I_N$ (the $N \times N$ identity matrix), $\lambda(W)$ and $\theta(W)$ attain a maximum value of 1. On the other hand, the choice $W = \frac{\mathbf{11}^\top}{N}$ attains the minimum of both, with $\lambda(W) = 0$ and $\theta(W) = \frac{1}{\sqrt{N}}$. 
\end{remark}%
\begin{remark}[Comparisons between density measures]
\label{remark:density_comparisons}
Since $\theta(W)$ depends only on pairwise interactions whereas $\lambda(W)$ utilizes the global structure of $W$, it is natural to expect that $\lambda(W)$ is a stronger measure of density. To see this more formally, let us suppose for simplicity that $W$ is a symmetric, irreducible matrix. All eigenvalues of $W$ are real, and suppose we write them as $\mu_1, \ldots, \mu_N$ so that $\mu_i^2 \ge \mu_j^2$ when $i \le j$. Since $W$ is row-stochastic, the Perron-Frobenius theorem implies that $\mathbf{1}$ is the unique eigenvector of $W$ with eigenvalue $\mu_1 = 1$. All other eigenvectors of $W$ are orthogonal to $\mathbf{1}$ by the symmetry of $W$, so {$\lambda(W)^2 = \mu_2^2$}. On the other hand, properties of the Frobenius norm imply that $\theta(W)^2 = \frac{1}{N} \sum_{i \in [N]} \mu_i^2 \le \mu_2^2 + 1/N = \lambda(W)^2 + 1/N$.
Hence, if $N$ is large and $\lambda(W)$ is small, $\theta(W)$ must be small as well.
\end{remark} 

\subsection{Robustness of the CMFA}
Our first main result establishes an analogue of Theorem \ref{thm:bw} that accounts for the spectral density of the interaction structure under identical agent behavior.

\begin{theorem}
\label{thm:deviation_rapidly_mixing}
Let $N$ be the size of the population and suppose that for all $i \in [N]$ and for all $\alpha, \beta \in \cS$, $r_i = 1$ and $\rho_i^{\alpha \beta} = \rho^{\alpha \beta}$. Let $Y_{av}(t)$ be the population average corresponding to the population process $\mathbf{Y}(t)$ with aggregation matrix $W$ and let $x(t)$ solve \eqref{eq:mean_field} with initial condition $x(0) = Y_{av}(0)$. There exist constants $a_1, b_1$ depending only on $\boldsymbol{\rho}, \phi, |\cS|$ such that for any time horizon $T > 0$ and $\epsilon > 0$,
\begin{equation}
\label{eq:rapidly_mixing_concentration}
\p \left( \sup\limits_{0 \le t \le T} \norm{ Y_{av}(t) - x(t)}_\infty > \revision{T} e^{a_1 T} \lambda(W) +  \epsilon \right) \le 2| \cS| \mathrm{exp} \left( -  \frac{N \epsilon^2}{e^{b_1 T}} \right).
\end{equation}
\end{theorem}

Most significantly, Theorem \ref{thm:deviation_rapidly_mixing} holds for \emph{any} population process, whereas Theorem \ref{thm:bw} only holds in the case $W = \mathbf{11}^\top / N$. Moreover, the deviation between the stochastic population process and the CMFA is 
$O(\lambda(W) + 1/\sqrt{N})$. 
The proof of Theorem \ref{thm:deviation_rapidly_mixing} is essentially due to a perturbation argument. Recalling that $\overline{\mathbf{Y}}(t) = \{ \overline{Y}_i(t) \}_{i \in [N]}$ is the collection of local population estimates, we have for any $\alpha \in \cS$ that 
$ \overline{\mathbf{Y}}^\alpha(t) - Y_{av}^\alpha(t) \mathbf{1} = W( \mathbf{Y}^\alpha(t) - Y_{av}^\alpha(t) \mathbf{1} )$. Since $\mathbf{1}^\top ( \mathbf{Y}^\alpha(t) - Y_{av}^\alpha(t) \mathbf{1}) = 0$ and $| Y_i^\alpha(t) - Y_{av}^\alpha(t) | \le 1$, it holds that $\| \overline{Y}^\alpha(t) - Y_{av}^\alpha(t) \mathbf{1} \|_2^2 / N \le \lambda(W)^2$.
In other words,
the difference between $\overline{Y}_i(t)$ and $Y_{av}(t)$ is $O ( \lambda(W))$ on average. It follows that if $\lambda(W)$ is small, most of the local population estimates are \emph{approximately equal} to the true population average $Y_{av}(t)$. Hence the stochastic dynamics resemble the case of homogeneous interactions where the local population estimates are \emph{equal} to $Y_{av}(t)$, and Theorem \ref{thm:deviation_rapidly_mixing} follows. We provide a formal proof in Section \ref{subsec:expanders}. 

\begin{remark}
Let $\{W_N \}_N$ be a sequence of row-stochastic matrices such that $W_N \in \reals^{N \times N}$ and consider a sequence $\{ \mathbf{Y}^N(t) \}_{N}$ such that $\mathbf{Y}^N(t)$ is a stochastic population process with interaction matrix $W_N$. If $\lambda(W_N) \to 0$, then Theorem \ref{thm:deviation_rapidly_mixing} recovers a {concentration} inequality of a similar form as Theorem \ref{thm:bw}. The property $\lambda(W_N) \to 0$ can be satisfied for a variety of interaction structures, such as random regular graphs with degree increasing in $N$ or complete graphs with $o(N)$ link failures per node 
(see Remarks \ref{remark:link_failures} and \ref{remark:expander}).
\end{remark}

\begin{remark}
In light of Remark \ref{remark:spectral_density_homogenous}, Theorem \ref{thm:deviation_rapidly_mixing} may be viewed as a {\it robust} version of Theorem \ref{thm:bw} in the following sense. If the underlying topology $(G,W)$ is complete but is perturbed to $(G', W')$ (due to the failure or deletion of inter-agent links, for instance), the mean-field characterization of the population process still holds as long as $\lambda(W')$ is close to $\lambda(W) = 0$. 
\end{remark}

\begin{remark}
\label{remark:time_scaling}
Theorem \ref{thm:deviation_rapidly_mixing} can be readily extended to handle the case where $r \neq 1$. To see why, let $\mathbf{Y}(t)$ be the original stochastic process and let $\mathbf{Y}'(t)$ be a version of $\mathbf{Y}(t)$ where the only difference is that $r = 1$. Since agent updates occur according to a Poisson process, we have $\mathbf{Y}(t) \stackrel{d}{=} \mathbf{Y}'(rt)$. We may therefore derive concentration results for $\mathbf{Y}(t)$ by applying Theorem \ref{thm:deviation_rapidly_mixing} to $\mathbf{Y}'(t)$ and scaling the time horizon appropriately.
\end{remark}

\subsection{Approximation by the NIMFA}

Our second set of results shows that when $\theta(W)$ is small -- a much looser condition than being spectrally dense -- then the MFA \eqref{eq:mean_field} serves as a good approximation for the stochastic population process over bounded time intervals. Unlike Theorems \ref{thm:bw} and \ref{thm:deviation_rapidly_mixing}, we consider the most general setup where the $r_i$'s and the $\rho_i^{\alpha \beta}$'s can be different across $i \in [N]$. Our main result is the following. 

\begin{theorem}
\label{thm:deviation}
Suppose that $\mathbf{y}(t)$ solves $\dot{\mathbf{y}} = \Phi( \mathbf{y})$ with initial condition $\mathbf{y}(0) = \mathbf{Y}(0)$ and that $r_{max} : = \max_{i \in [N]} r_i < 1$. 
\revision{Let $\cR$ be a bound on the largest column sum of $W$.}
Then there exist constants $a_2, b_2, c_2$ depending only on $\boldsymbol{\rho}, | \cS |, r_{max}, \cR$ such that for all $T, \epsilon > 0$, 
\begin{align}
\p & \left( \sup\limits_{0 \le t \le T} \| Y_{av}(t) - y_{av}(t) \|_\infty > 10 \cR | \cS | e^T \theta(W) + \epsilon \hspace{-0.05cm} \right) \hspace{-0.05cm}  \le  h(T, \epsilon) \cdot \mathrm{exp} \left(  - \frac{N \epsilon^2}{ e^{a_2  T} }  \right) \nonumber 
\end{align}
where $h(T, \epsilon) : = (   | \log \epsilon | e^{ b_2  T}/\epsilon  )^{c_2 ( T + | \log \epsilon |)}$. 
\end{theorem}
Theorem \ref{thm:deviation} can be seen as a generalization of Theorem \ref{thm:bw} in the sense that we obtain a similar non-asymptotic large-deviations probability bound for the setting of general agent behavior and heterogeneous interactions, although the constants involved are different. Importantly, Theorem \ref{thm:deviation} holds for any interaction structure $W$ and the approximation error of the NIMFA is $O( \theta(W) + 1/\sqrt{N})$. The proof of Theorem \ref{thm:deviation} is quite involved, and we defer the details to Section \ref{subsec:concentration_general}. At a high level, we study martingales associated to the process $Y_{av}(t) - y_{av}(t)$ and apply standard concentration inequalities to obtain an exponentially decaying probability bound. The devil is in the details, however, and our analysis requires a careful and tight control of the dependencies between agent behaviors. 

\begin{remark}[Random walk matrices]
Suppose that $W$ is a random walk matrix as described in Remark \ref{remark:random_walk}. We may bound the maximum column sum as 
$$
\max\limits_{j \in [N]} \sum\limits_{i \in [N]} w_{ij}  = \max\limits_{j \in [N]} \sum\limits_{i \in [N]} \frac{ \mathbf{1}( w_{ij} > 0) }{ \mathrm{deg}(i)} \le \frac{d_{max}}{d_{min}} = : \cR,
$$
where $d_{max}$ is the maximum degree and $d_{min}$ is the minimum degree over all agents. 
Consequently, if we consider a sequence $\{\mathbf{Y}^N(t) \}_N$ of stochastic population processes where $\mathbf{Y}^N(t)$ has interaction matrix $W_N$, $\theta(W_N) \to 0$ and the maximum column sum of the $W_N$'s are uniformly bounded by some fixed $\cR$, then Theorem \ref{thm:deviation} recovers a concentration inequality of a similar form as Theorem \ref{thm:bw} for large $N$. 
\end{remark}

\begin{remark}
\label{remark:deviation}
The condition $r_{max} : = \max_{i \in [N]} r_i < 1$ is made for convenience, and Theorem \ref{thm:deviation} can be readily extended to handle the case of general $r_{max}$. To see why, let $\mathbf{Y}'(t)$ be a version of the original process with interaction rate for agent $i$ given by $r_i' : = r_i / (2 r_{max})$ so that $r'_i \le 1/2 < 1$ for $\mathbf{Y}'(t)$. Following the same arguments of Remark \ref{remark:time_scaling}, $\mathbf{Y}(t) \stackrel{d}{=} \mathbf{Y}'(2 r_{max} t)$, hence concentration for $Y_{av}(t)$ can be obtained by applying Theorem \ref{thm:deviation} to $\mathbf{Y}'(t)$ and scaling the time horizon appropriately.
\end{remark}

\begin{remark}[Comparison to \cite{horvath2022accuracy}]
\label{remark:comparison}
After an initial draft of our work was posted online \cite{sridhar_kar_arxiv_2021}, we came across a recent work by Horv\'{a}th and Keliger that proved a version of Theorem \ref{thm:deviation} using different techniques which holds for higher-order interactions (i.e., defined by a hypergraph); see \cite[Theorem 4]{horvath2022accuracy}. Their probability bound is considerably weaker then the exponential bound we provide in Theorem \ref{thm:deviation}. Since our methods are quite generic, we conjecture that our exponential bounds may translate to the setting of higher-order interactions, though we leave this to future work.
\end{remark}

\section{Choosing the right mean-field approximation}
\label{sec:choosing}
 Although our results show that the NIMFA is generally more accurate than the CMFA, simulating the NIMFA can be computationally intractable for large populations. A natural question, therefore, is whether our approximation results are \emph{tight}: are there situations where $W$ is not spectrally dense, yet the CMFA provides a good approximation for $Y_{av}(t)$? Our simulations indicate that if agents choose their initial states independently, the CMFA serves as a good approximation to the population process, even for interaction structures that are not spectrally dense. On the other hand, we provide examples of more structured initial conditions, depending on $W$, for which $Y_{av}(t)$ and the NIMFA deviate significantly from the CMFA, both qualitatively and quantitatively. 

\subsection{Example: Coordination games on nearest-neighbor graphs}
\label{subsec:simulations_setup}

Consider the 2-action population game $\cU(x_1, x_2) = (\cU^1(x_1, x_2), \cU^2(x_1, x_2)) = (x_1, 2x_2)$, which is sometimes known as a \emph{coordination game} \cite[Chapter 2]{sandholm2010popgames}. We assume for simplicity that agent behavior is homogeneous (i.e., $r_i = 1$ and $\rho_i^{\alpha \beta} = \rho^{\alpha \beta}$ for all $i \in [N]$) and that agents evolve via a \emph{logit choice protocol} \cite{sandholm2010popgames}, defined for $\alpha, \beta \in \{1,2 \}$ by $\rho^{\alpha \beta} (x) := \mathrm{exp} ( \eta^{-1} \cU^\beta(x) ) / ( \mathrm{exp} ( \eta^{-1} \cU^1(x)) + \mathrm{exp} ( \eta^{-1} \cU^2(x)) )$ with noise level $\eta = 0.1$. We assume that the interaction structure between agents is induced by a \emph{nearest neighbor graph}, the construction of which is outlined below. 

\begin{definition}[Nearest neighbor graphs]
We say $G$ is a nearest neighbor graph with $N$ nodes and density $\gamma \in [0,1]$ if $G$ is constructed as follows: first, $N$ vertices are placed at equidistant locations on the unit circle. Then, letting $d$ be the closest even number to $\gamma N$, each vertex forms an edge with the $d$ closest neighbors in $G$. 
\end{definition}

We assume that the aggregation matrix $W_{N, \gamma}$ is the random walk matrix corresponding to a nearest neighbor graph on $N$ vertices and density $\gamma$. 
It turns out that $W_{N, \gamma}$ is {\it not} spectrally sparse: it holds that $\lim_{N \to \infty} \lambda(W_{N, \gamma}) = \sin(\pi \gamma) / (\pi \gamma) > 0$. This
essentially follows from the observation that $W_{N, \gamma}$ is a circulant matrix,
hence the eigenvalues can be explicitly computed. As the calculations are quite elementary, we leave the details to the interested reader due to space constraints.
Since $\lambda(W_{N, \gamma})$ is bounded away from zero, Theorem \ref{thm:deviation_rapidly_mixing} does not provide any meaningful results about the concentration of $Y_{av}(t)$ around the CMFA. However, since $W_{N, \gamma}$ corresponds to a regular graph, Remark \ref{remark:random_walk} implies that $\theta(W) = \sqrt{1/d} = ( \gamma N)^{-1/2}$. Theorem \ref{thm:deviation} therefore implies that $Y_{av}(t)$ is well-approximated by the NIMFA. 

In our simulations, we consider two types of initial conditions, which we refer to as {\it clustered initial conditions} and {\it random initial conditions}. In both, we let 80\% of agents initially play action 1, and 20\% of agents initially play action 2, but the {\it locations} of these agents differ in the two types of initial conditions.
When simulating {\it random} initial conditions, each agent independently chooses action 1 initially with probability $0.8$, otherwise action 2 is chosen. Since neighborhoods are large, the law of large numbers implies that each neighborhood should approximately have an 80\% fraction of vertices playing action 1 and a 20\% fraction of vertices playing action 2, which closely aligns with the true population average. We therefore expect that the CMFA will be quite accurate for random initial conditions, at least in the early stages of evolution.
On the other hand, the idea behind clustered initial conditions is to provide an example where the local population estimates deviate significantly from the average, which would imply that the CMFA may be inaccurate. Specifically, the agents who play action 2 are chosen to be consecutive vertices on the unit circle -- hence the name {\it clustered initial conditions.} As a result, vertices in this ``cluster'' will have a high neighborhood fraction of vertices playing action 2, whereas vertices far from the cluster will have a high neighborhood fraction of vertices playing action 1. The local population states can vary significantly from the average, at least in the early stages of the evolution, which will cause significant deviations from the CMFA.

\subsection{Simulation results}

We simulated the population process induced by $r_i = 1$, $\boldsymbol{\rho}$ given by the logit choice dynamics with noise level $\eta = 0.1$, and $W = W_{N, \gamma}$ with $N = 1000$ and various values of $\gamma$.

\paragraph{Random initial conditions} We simulated $Y_{av}(t)$ for $\gamma = 0.2, 0.5, 0.8$ with random initial conditions and compared their trajectories to the CMFA \revision{and the NIMFA for $\gamma = 0.2$}; see Figure \ref{subfig:random}. Notably, even though $W$ is not spectrally dense, the CMFA captures the evolution of $Y_{av}(t)$ \revision{as well as the NIMFA}. Somewhat surprisingly, the CMFA \revision{accurately approximates both the NIMFA and $Y_{av}(t)$} over longer time periods, suggesting that the CMFA is a {\it stable trajectory} of the NIMFA.

\paragraph{Clustered initial conditions} We simulated the resulting population process for \revision{$\gamma = 0.2$}, the corresponding NIMFAs, as well as the CMFA; our results can be found in Figure \ref{fig:structured_initial_conditions}. As predicted by Theorem \ref{thm:deviation}, $Y_{av}(t)$ is well-approximated by the corresponding NIMFA. Moreover, 
since the empirical distribution of many neighborhoods deviate significantly from the population average for these initial conditions,
the NIMFA and the CMFA {can} deviate significantly {in transient stages, a illustrated by Figure \ref{subfig:structured_0.2}.}
{Interestingly, the NIMFA and CMFA also have different steady-state behavior under clustered initial conditions (see Figure \ref{subfig:asymptotics}), showing that the topological structure of initial conditions can have a considerable long-term effect on the system.}
A deeper understanding of {such fundamental} differences between the NIMFA and the CMFA is an important avenue for future work. 

\begin{figure}[t]
    \centering
    \begin{subfigure}{0.33\textwidth}
        \centering
        \includegraphics[width=\textwidth]{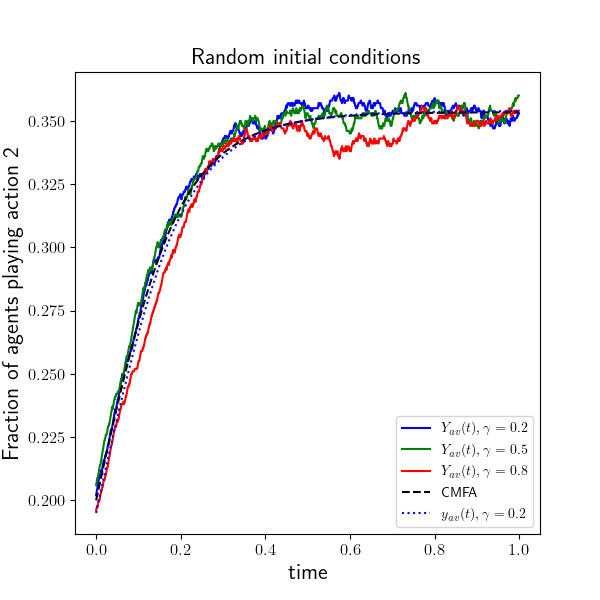}
        \caption{Random initial conditions}
        \label{subfig:random}
    \end{subfigure}%
    \begin{subfigure}{0.33 \textwidth}
        \centering
        \includegraphics[width=\textwidth]{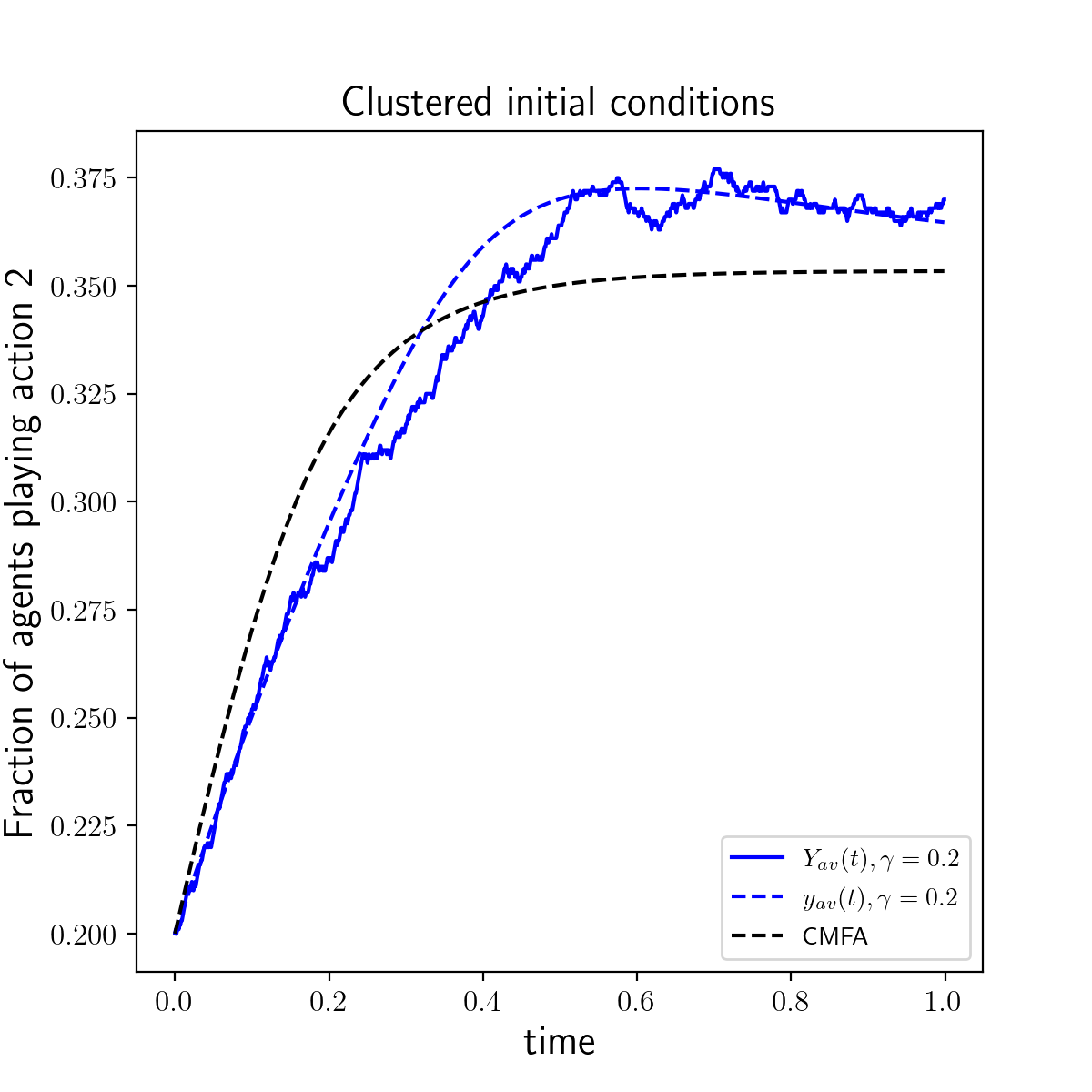}
        \caption{{$\gamma = 0.2$}}
        \label{subfig:structured_0.2}
    \end{subfigure}%
    \begin{subfigure}{0.33\textwidth}
        \centering
        \includegraphics[width=\textwidth]{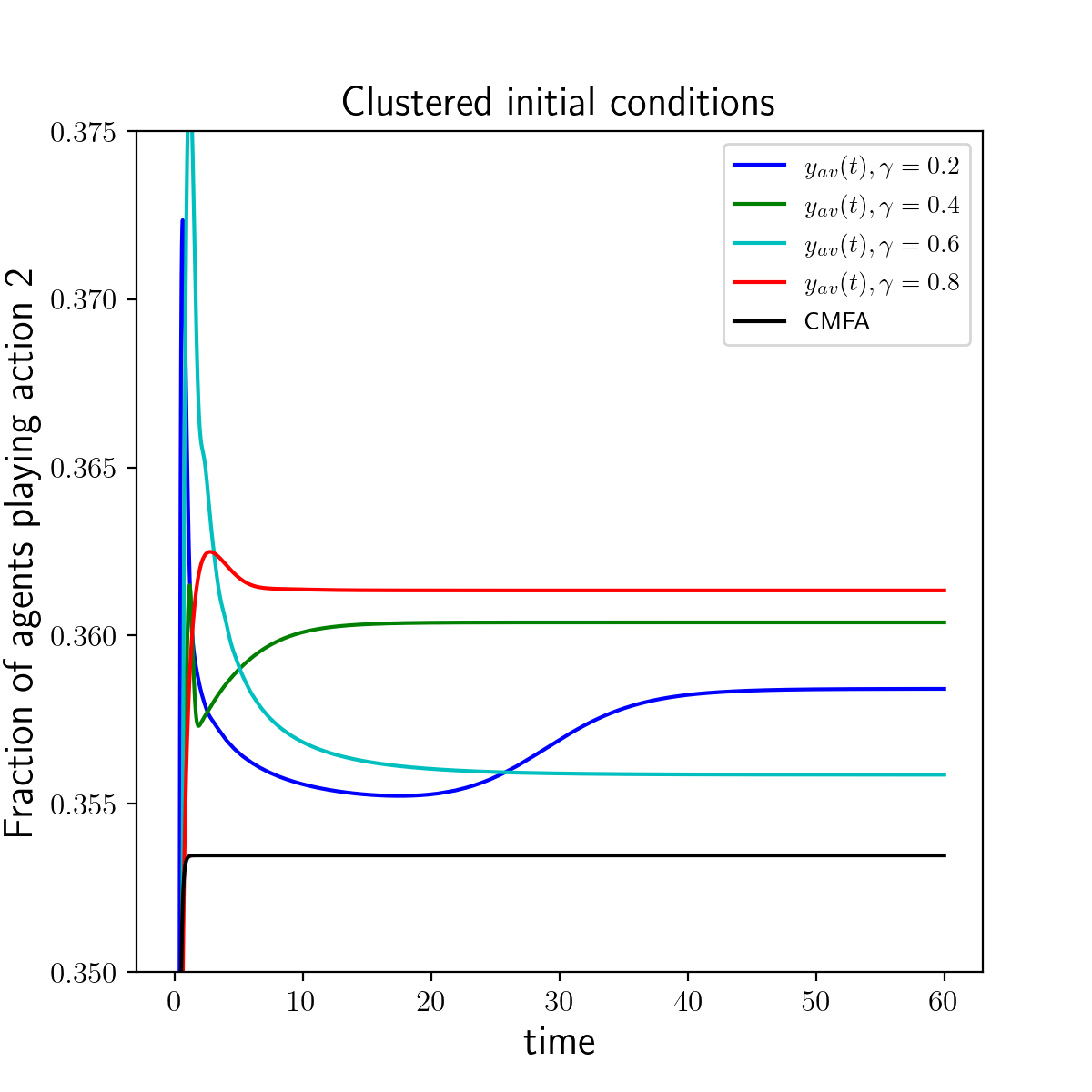}
        \caption{{Asymptotics of the NIMFA}}
        \label{subfig:asymptotics}
    \end{subfigure}%
    \caption{Simulations of $Y_{av}(t)$, the CMFA and the NIMFA. Fig. \ref{subfig:random} illustrates random initial conditions for $\gamma \in \{ 0.2, 0.5, 0.8\}$. 
    {Fig. \ref{subfig:structured_0.2} illustrates the impact of clustered initial conditions on $Y_{av}(t)$ and the NIMFA.}
    {Fig. \ref{subfig:asymptotics} shows the temporal asymptotics of the NIMFA with clustered initial conditions with $N=200$. We take a smaller value of $N$ for Fig. \ref{subfig:asymptotics} due to the computational overhead in solving the NIMFA over large time horizons.}}
    \label{fig:structured_initial_conditions}
\end{figure}%

\section{Proofs of mean-field concentration results}

In this section, we prove Theorems \ref{thm:deviation_rapidly_mixing} and \ref{thm:deviation}. While the high-level strategy follows the methods of Bena\"{i}m and Weibull \cite{BenaimWeibull}, there are many subtle differences in the details caused by the heterogeneities induced by $W$ which require novel technical workarounds. In Section \ref{subsec:expanders}, we prove Theorem \ref{thm:deviation_rapidly_mixing}, which also serves as a gentle introduction to our methods. Finally, we prove our more general result, Theorem \ref{thm:deviation}, in Section \ref{subsec:concentration_general}. 

Before moving to our analysis, we introduce some notation. For a vector $v \in \reals_{\ge 0}^N$, define $Y_v(t) : = \sum_{i \in [N]} v_i Y_i(t)$, $y_v(t) : = \sum_{i \in [N]} v_i y_i(t)$, $\Phi_v(\cdot) : = \sum_{i \in [N]} v_i \Phi_i(\cdot)$ and $M_v(t) : = Y_v(t) - Y_v(0) - \int_0^t \Phi_v(\mathbf{Y}(s)) ds$. In the special case $v = \mathbf{1}/N$, we use the subscript ``$av$" (i.e., $Y_{av}(t)$) to indicate that we are taking an average over all agents. 

\subsection{Concentration for spectrally dense interactions: Proof of Theorem \ref{thm:deviation_rapidly_mixing}}
\label{subsec:expanders}

In this section we prove Theorem \ref{thm:deviation_rapidly_mixing}, which follows from a few intermediate results. Recall that for the purposes of this section alone, we assume that $r_i = 1$ and $\rho_i^{\alpha \beta} = \rho^{\alpha \beta}$ for all $i \in [N]$ and $\alpha, \beta \in \cS$. Our first intermediate result shows that the expected rate of change for $Y_{av}(t)$ is approximately equal to the rate of change for the classical mean-field ODE. 

\begin{lemma}
\label{lemma:expander_perturbation_bound}
For any $t \ge 0$, 
$\| \Phi_{av}(\mathbf{Y}(t)) - \phi( Y_{av}(t) ) \|_\infty \le 2L_\rho | \cS |^{3/2} \lambda(W)$.
\end{lemma}

\begin{proof}
For brevity, define $f_i^{\alpha \beta}(t) : = \rho^{\alpha \beta} ( \overline{Y}_i(t)) - \rho^{\alpha \beta} ( Y_{av}(t))$. For any $\alpha$, we can write $\Phi_{av}^\alpha( \mathbf{Y}(t)) - \phi^\alpha ( Y_{av}(t)) = \frac{1}{N} \sum_{i \in [N]} \sum_{\beta \in \cS} ( Y_i^\beta(t) f_i^{\beta \alpha}(t) - Y_i^\alpha(t) f_i^{\alpha \beta}(t) )$. The absolute value of $\Phi_{av}^\alpha(\mathbf{Y}(t)) - \phi^\alpha(Y_{av}(t))$ can then be bounded as
\begin{align*}
& | \Phi_{av}^\alpha(t) - \phi^\alpha( Y_{av}(t)) |  \stackrel{(a)}{\le} \frac{1}{N} \sum\limits_{i = 1}^N \sum\limits_{\beta \in \cS} \left( |f_i^{\beta \alpha}(t) | + | f_i^{\alpha \beta}(t) | \right) \\
&  \stackrel{(b)}{\le} \frac{2L_\rho | \cS |}{N} \sum\limits_{i = 1}^N  \| \overline{Y}_i(t) - Y_{av}(t) \|_\infty  \stackrel{(c)}{\le} \frac{2L_\rho | \cS |}{N} \sum\limits_{i = 1}^N \sqrt{ \sum\limits_{\beta \in \cS} \left( \overline{Y}_i^\beta(t) - Y_{av}^\beta(t) \right)^2 } \\
& \stackrel{(d)}{\le} 2L_\rho | \cS |\sqrt{ \frac{1}{N} \sum\limits_{i = 1}^N \sum\limits_{\beta \in \cS} \left( \overline{Y}_i^\beta(t) - Y_{av}^\beta(t) \right)^2 }  = 2L_\rho | \cS | \sqrt{ \frac{1}{N} \sum\limits_{\beta \in \cS} \left \| W ( \mathbf{Y}^\beta(t) - Y_{av}^\beta(t) \mathbf{1} ) \right \|_2^2 } \\
& \le 2L_\rho | \cS | \sqrt{ \frac{1}{N} \sum\limits_{\beta \in \cS} \lambda(W)^2 \| \mathbf{Y}^\beta(t) - Y_{av}^\beta(t) \mathbf{1} \|_2^2  }  \stackrel{(e)}{\le} 2L_\rho | \cS |^{3/2} \lambda(W).
\end{align*}
Above, $(a)$ uses the triangle inequality and that $| Y_i^\alpha(t)  | \le 1$, $(b)$ uses the bound $|f_i^{\alpha \beta} (t) | \le L_\rho \| \overline{Y}_i(t) - Y_{av}(t) \|_\infty$ which holds since all the $\rho^{\alpha \beta}$'s are $L_\rho$-Lipschitz (see Assumption \ref{as:rho}), $(c)$ uses $\| \cdot \|_\infty \le \| \cdot \|_2$, $(d)$ is due to Jensen's inequality and $(e)$ holds since $| Y_i^\beta(t) - Y_{av}^\beta(t) | \le 1$, implying that $\| \mathbf{Y}^\beta(t) - Y_{av}^\beta(t) \mathbf{1} \|_2^2 \le N$. 
\end{proof}

\revision{Next, we study the tail of  $M_{av}(t) := Y_{av}(t) - Y_{av}(0) - \int_0^t \Phi_{av}( \mathbf{Y}(s)) ds$.}

\begin{lemma}
\label{lemma:Mav_concentration}
Let $N, T$ be given. For any $\epsilon \in (0,1)$, 
$$
\p \left( \sup\limits_{0 \le t \le T} \| M_{av}(t) \|_\infty > \epsilon \right) \le 2 | \cS | \mathrm{exp} \left( - \frac{N \epsilon^2}{4 e^T} \right).
$$
\end{lemma}

Bena\"{i}m and Weibull proved Lemma \ref{lemma:Mav_concentration} for the special case $W = \mathbf{11}^\top / N$, but their proof in \cite{BenaimWeibull} crucially used the Markovianity of $Y_{av}(t)$ in this case. However, $Y_{av}(t)$ is not Markov in general (only the full process $\mathbf{Y}(t)$ is Markov).
Nevertheless, standard martingale concentration inequalities can be used to prove Lemma \ref{lemma:Mav_concentration} for general row-stochastic $W$; we defer the proof to Appendix \ref{sec:martingale_properties} for the interested reader.

\begin{proof}[Proof of Theorem \ref{thm:deviation_rapidly_mixing}]
Let $0 \le t \le T$. 
\revision{Using the relation $Y_{av}(t) - x(t) = Y_{av}(t) - Y_{av}(0) - \int_0^t \phi(x(s)) ds$, we can add and subtract terms to write
\[
Y_{av}(t) - x(t) = \int_0^t \phi( Y_{av}(s)) - \phi( x(s)) + ( \Phi_{av}(\mathbf{Y}(s)) - \phi(Y_{av}(s)) ) ds + M_{av}(t).
\]
We proceed by taking the infinity norm on both sides. Since $\phi$ is $L_\phi$-Lipschitz, we have that $\| \phi( Y_{av}(s) ) - \phi(x(s)) \|_\infty \le L_\phi \| Y_{av}(s) - x(s) \|_\infty$. 
Using Lemma \ref{lemma:expander_perturbation_bound}, we have that $\| \Phi_{av}( \mathbf{Y}(s)) - \phi(Y_{av}(s)) \|_\infty \le L \lambda (W)$ with $L : = 2 L_\rho |\cS|^{3/2}$.
Together, these inequalities imply that
\begin{align}
\label{eq:yav_error_bound}
\| Y_{av}(t) - x(t) \|_\infty & \le L_\phi \int_0^t \| Y_{av}(s) - x(s) \|_\infty + LT \lambda(W) + \sup\limits_{0 \le t \le T} \| M_{av}(t) \|_\infty.
\end{align}
Applying Gr\"{o}nwall's inequality to \eqref{eq:yav_error_bound} yields
\[
\sup\limits_{0 \le t \le T} \| Y_{av}(t) - x(t) \|_\infty \le \left( L T \lambda(W) + \sup\limits_{0 \le t \le T} \| M_{av}(t) \|_\infty \right) e^{ L_\phi T}.
\]
Rearranging terms, we have that 
\[
\sup_{0 \le t \le T} \| M_{av}(t) \|_\infty \ge e^{- L_\phi T} \left( \sup_{0 \le t \le T} \| Y_{av}(t) - x(t) \|_\infty \right) - L T \lambda(W).
\]
In particular, if $\sup_{0 \le t \le T} \| Y_{av}(t) - x(t) \|_\infty \ge L T e^{L_\phi T} \lambda(W) + \epsilon$, then we have that $\sup_{0 \le t \le T} \| M_{av}(t) \|_\infty \ge \epsilon / e^{L_\phi T}$. We conclude by using Lemma \ref{lemma:Mav_concentration} to bound the probability of the latter event.}
\end{proof}

\subsection{Concentration for general processes: Proof of Theorem \ref{thm:deviation}}
\label{subsec:concentration_general}

In this section we prove Theorem \ref{thm:deviation} under general conditions (i.e., general $r_i, \rho_i, W$). While the general proof strategy is the same as that of Theorem \ref{thm:deviation_rapidly_mixing}, many of the intermediate steps no longer hold in the general setting, which requires us to develop new techniques to prove concentration.

Define, for a non-negative matrix $P$, the process 
$A(P, t) : =  \sum_{i \in [N]} \| Y_{\mathbf{p}_i}(t) - y_{\mathbf{p}_i}(t) \|_\infty / N$,
where $\mathbf{p}_i$ is the $i$th row vector of $P$. Notice that by setting $P = \mathbf{11}^\top / N$ we recover the error process of interest. An important observation is that $A(P, t)$ captures the average behavior of a collection of \emph{linear projections} of the error process $\mathbf{Y}(t) - \mathbf{y}(t)$ onto $\reals$.

We proceed by following the proof of Theorem \ref{thm:deviation_rapidly_mixing}, with the goal of bounding $A(P,t)$ for a given sub-stochastic matrix $P$. To this end, we can write
$$
Y_{\mathbf{p}_i}(t) - y_{\mathbf{p}_i}(t) = \int_0^t \Phi_{\mathbf{p}_i}( \mathbf{Y}(s) ) - \Phi_{\mathbf{p}_i}( \mathbf{y}(s)) ds + M_{\mathbf{p}_i}(t).
$$
Taking the infinity norm of both sides and applying the triangle inequality shows that
$$
\| Y_{\mathbf{p}_i}(t) - y_{\mathbf{p}_i}(t) \|_\infty \le \int_0^t \| \Phi_{\mathbf{p}_i}(\mathbf{Y}(s)) - \Phi_{\mathbf{p}_i}( \mathbf{y}(s) ) \|_\infty ds + \| M_{\mathbf{p}_i}(t) \|_\infty.
$$
We may now average over $i \in [N]$ to obtain
\begin{equation}
\label{eq:A_recursion}
A(P, t) \le \int_0^t  \frac{1}{N} \sum\limits_{i \in [N]} \| \Phi_{\mathbf{p}_i}( \mathbf{Y}(s) ) - \Phi_{\mathbf{p}_i}( \mathbf{y}(s)) \|_\infty ds + \frac{1}{N} \sum\limits_{i \in [N]} \| M_{\mathbf{p}_i}(t) \|_\infty.
\end{equation}
To apply the same proof strategy as Theorem \ref{thm:deviation_rapidly_mixing}, two key ingredients are needed: (1) a bound on the integrand in \eqref{eq:A_recursion} in terms of projections of $\mathbf{Y}(t) - \mathbf{y}(t)$ related to $P$, and (2) a concentration inequality for the term on the right hand side of \eqref{eq:A_recursion}.

Since $\Phi_{\mathbf{p}_i}$ is a nonlinear function, achieving (1) is a non-trivial task. Indeed, if $\Phi_{\mathbf{p}_i}$ is a generic $L$-Lipschitz function, we can only expect to bound the integrand of \eqref{eq:A_recursion} by $L \| \mathbf{Y}(s) - \mathbf{y}(s) \|$ (where $\| \cdot \|$ is an appropriately chosen norm on $\Delta(\cS)^N$) which will generally not be close to zero. However, by exploiting the structure of $\Phi_{\mathbf{p}_i}$, we can show that the integrand of \eqref{eq:A_recursion} is small, provided that a well-chosen collection of (deterministic) linear projections of $\mathbf{Y}(t) - \mathbf{y}(t)$ are also small on average. This is formalized in the following lemma. See Section \ref{sec:gamma} for the proof.

\begin{lemma}
\label{lemma:Gamma_properties}
Fix a sub-stochastic matrix $P \in \reals^{N \times N}$ with maximum column sum at most $\cR$ and a time horizon $T \ge 0$. Suppose further that $W$ has a maximum column sum at most $\cR$. There exists a deterministic set of non-negative matrices $\Gamma(P, T)$ such that the following hold: 
\begin{enumerate}

\item \label{item:Gamma_basic_properties}
For all $Q \in \Gamma(P, T)$, $\theta(Q) \le \max \{ \theta(P), \theta(W) \}$ and $Q$ has a maximum column sum bounded by $\cR$.

\item \label{item:Gamma_lipschitz}
For all $t \in [0,T]$, 
$$
\frac{1}{N} \sum\limits_{i \in [N]} \| \Phi_{\mathbf{p}_i}( \mathbf{Y}(t)) - \Phi_{\mathbf{p}_i} ( \mathbf{y}(t)) \|_\infty \le  (   L_\rho \cR + 1) | \cS |^2 \sup\limits_{Q \in \Gamma(P, T)} A(Q, t),
$$
where $L_\rho$ is defined in Assumption \ref{as:rho}.

\end{enumerate}
\end{lemma}

We remark that a consequence of the non-linearity of $\Phi_{\mathbf{p}_i}$ is that the set $\Gamma(P, T)$ also depends on the trajectory $\{ \mathbf{y}(t) \}_{0 \le t \le T}$ in a non-linear manner
(see \eqref{eq:gamma_p_t} in Section \ref{sec:gamma} for the explicit construction of $\Gamma(P,T)$).
Nevertheless, since $\mathbf{y}(t)$ is a deterministic quantity, we shall see that the dependence on $\{\mathbf{y}(t) \}_{0 \le t \le T}$ does not significantly complicate the analysis of $\sup_{Q \in \Gamma(P, T) } A(Q, t)$. 

In the remainder of the section, set $L : = (L_\rho \cR + 1) | \cS |^2$. Applying Lemma \ref{lemma:Gamma_properties} to \eqref{eq:A_recursion} shows that, for $t \in [0,T]$, 
\begin{equation}
\label{eq:A_recursion_v2}
A(P, t) \le L  \int_0^t \sup\limits_{Q \in \Gamma(P, T)} A(Q, s) ds + \frac{1}{N} \sum\limits_{i \in [N]} \| M_{\mathbf{p}_i}(t) \|_\infty.
\end{equation}
To follow the proof of Theorem \ref{thm:deviation_rapidly_mixing}, the left hand side of \eqref{eq:A_recursion_v2} must be the same as the integrand on the right hand side of \eqref{eq:A_recursion_v2}. To this end, we construct a convenient subset of sub-stochastic matrices, denoted by $\Gamma_\infty$, for which $\sup_{P \in \Gamma_\infty} A(P, t)$ satisfies a recursive inequality that can be handled by Gr\"{o}nwall's inequality. 

\begin{definition}[The set $\Gamma_\infty$]
\label{def:Gamma_infinity}
Construct a sequence of sets $\{\Gamma_k\}_{k \ge 0}$ such that $\Gamma_0 : = \{ \mathbf{11}^\top / N\}$ and $\Gamma_k : = \bigcup_{P \in \Gamma_{k-1}} \Gamma(P, T)$ for $k \ge 1$. We define $\Gamma_\infty : = \bigcup_{k \ge 0} \Gamma_k$. 
\end{definition}

By the construction of $\Gamma_\infty$, we have that $P \in \Gamma_\infty \Rightarrow \Gamma(P, T) \subset \Gamma_\infty$. For $P \in \Gamma_\infty$, \eqref{eq:A_recursion_v2} implies that for $t \in [0,T]$,
$$
A(P, t) \le L  \int_0^t \sup\limits_{Q \in \Gamma_\infty} A(Q, s) ds + \frac{1}{N} \sum\limits_{i \in [N]} \| M_{\mathbf{p}_i}(t) \|_\infty. 
$$
Maximizing both sides over $P \in \Gamma_\infty$ shows that, for $t \in [0,T]$,
\begin{equation}
\label{eq:thm_concentration_gronwall_setup}
\sup\limits_{P \in \Gamma_\infty} A(P, t) \le L \int_0^t \sup\limits_{P \in \Gamma_\infty} A(P, s) ds + \sup\limits_{P \in \Gamma_\infty} \frac{1}{N} \sum\limits_{i \in [N]} \| M_{\mathbf{p}_i}(t) \|_\infty.
\end{equation}
We are now in a position to apply Gr\"{o}nwall's inequality to $\sup_{P \in \Gamma_\infty} A(P, t)$; we go through the details in the formal proof below. The final missing component of the proof of Lemma \ref{lemma:main_concentration} is a concentration inequality for the second term on the right hand side of \eqref{eq:thm_concentration_gronwall_setup}. The following lemma establishes this for a \emph{fixed} $P$. 

\begin{lemma}
\label{lemma:main_concentration}
Let $N, T, \epsilon$ be given. If $P$ is a sub-stochastic matrix with maximum column sum at most $\cR$ and satisfies $\theta(P) \le \theta(W)$, then 
$$
\p \left( \sup\limits_{0 \le t \le T} \frac{1}{N} \sum\limits_{i \in [N]} \left \| M_{\mathbf{p}_i}(t) \right \|_\infty \ge 10 \cR | \cS | e^T \theta(W) + \epsilon \right) \le  | \cS | \mathrm{exp} \left( - \frac{N \epsilon^2}{ 160 \cR | \cS |^2 e^T } \right).
$$
\end{lemma}

We remark that Lemma \ref{lemma:main_concentration} is perhaps the most technically involved result of this paper. Unlike the process $M_{av}(t)$, which is a martingale with uniformly bounded jumps, the process $\frac{1}{N} \sum_{i \in [N]} \| M_{\mathbf{p}_i}(t) \|_\infty$ is \emph{not} a martingale and has jumps which depend on the agents who update at a particular point in time. We therefore are required to perform a careful and tight analysis of the process to derive Lemma \ref{lemma:main_concentration}. We defer the details to Section \ref{sec:martingale_tail_inequalities}. 

A natural way to use Lemma \ref{lemma:main_concentration} to study the supremum of the martingale terms over $P \in \Gamma_\infty$ is to take a union bound over elements of $\Gamma_\infty$. Unfortunately, $\Gamma_\infty$ generally has infinitely many elements. To get around this issue, we bound the covering number of $\Gamma_\infty$. See Section \ref{sec:gamma} for the proof.

\begin{lemma}
\label{lemma:covering_number}
For every $\delta > 0$, there exists $\widetilde{\Gamma}_\delta \subset \Gamma_\infty$ such that the following hold: 
\begin{enumerate}
    \item For every $P \in \Gamma_\infty$, there exists $Q \in \widetilde{\Gamma}_\delta$ such that $\frac{1}{N} \sum_{i,j \in [N]} | p_{ij} - q_{ij} | \le \delta$; 
    \item $|\widetilde{\Gamma}_\delta | \le (c T | \log \delta | / \delta )^{c' | \log \delta |}$, where $c = c( | \cS |, r_{max}, L_\rho)$ and $c' = c'(r_{max})$.
\end{enumerate}
\end{lemma}

We can now put together all our intermediate results to prove the theorem. 

\begin{proof}[Proof of Theorem \ref{thm:deviation}]
Applying Gr\"{o}nwall's inequality to \eqref{eq:thm_concentration_gronwall_setup}, we obtain 
\begin{equation}
\label{eq:deviation_gronwall_result}
\sup\limits_{\substack{0 \le t \le T ,  P \in \Gamma_\infty}} A(P, t) \le \left( \sup\limits_{\substack{0 \le t \le T ,  P \in \Gamma_\infty}} \frac{1}{N} \sum\limits_{i \in [N]} \| M_{\mathbf{p}_i}(t) \|_\infty \right) e^{L T},
\end{equation}
where we recall that $L : = (L_\rho \cR + 1) | \cS|^2$.
Let $\delta > 0$. To handle the supremum of the martingale terms on the right hand side, we will first replace the supremum over the infinite set $\Gamma_\infty$ with the finite set $\widetilde{\Gamma}_\delta$ (defined in Lemma \ref{lemma:covering_number}) and then take a union bound over elements of $\widetilde{\Gamma}_\delta$. To this end, first notice that we have the bound
\[
\| M_i(t) \|_\infty \le \| Y_i(t) - Y_i(0) \|_\infty + \int_0^t \| \Phi_i( \mathbf{Y}(s)) \|_\infty ds \le 1 + t.
\]
For any two matrices $P, Q \in \reals_{\ge 0}^{N \times N}$, it follows that
\begin{align}
& \left | \frac{1}{N} \sum\limits_{i \in [N]} \| M_{\mathbf{p}_i}(t) \|_\infty - \frac{1}{N} \sum\limits_{i \in [N]} \| M_{\mathbf{q}_i}(t) \|_\infty \right|  \le \frac{1}{N} \sum\limits_{i \in [N]} \| M_{\mathbf{p}_i}(t) - M_{\mathbf{q}_i}(t) \|_\infty \nonumber \\
\label{eq:M_P_Q_bound}
& \hspace{1cm} \le \frac{ \max_{i \in [N]} \| M_i(t) \|_\infty }{N} \sum\limits_{i \in [N]} \| \mathbf{p}_i - \mathbf{q}_i \|_1 \le \frac{1 + t}{N} \sum_{i,j \in [N]} | p_{ij} - q_{ij} |.
\end{align}
We are now ready to put everything together to bound the tail of $\sup_{P \in \Gamma_\infty} A(P, t)$. 
Set $\delta := \frac{\epsilon}{2 e^{(L+1)T}} \le \frac{\epsilon}{2 (T + 1) e^{L  T}}$. It holds that
\begin{align*}
\p & \left( \sup\limits_{\substack{0 \le t \le T,  P \in \Gamma_\infty}} A(P, t) \ge 10 \cR | \cS | e^{(L + 1) T} \theta(W) + \epsilon \right) \\
& \hspace{2cm} \stackrel{(a)}{\le} \p \left( \sup\limits_{\substack{0 \le t \le T,  P \in \Gamma_\infty}} \frac{1}{N} \sum\limits_{i \in [N]} \| M_{\mathbf{p}_i}(t) \|_\infty \ge 10 \cR | \cS | e^T \theta(W) + \frac{\epsilon}{e^{L T}} \right) \\
& \hspace{2cm} \stackrel{(b)}{\le} \p \left( \sup\limits_{\substack{0 \le t \le T , P \in \widetilde{\Gamma}_\delta} } \frac{1}{N} \sum\limits_{i \in [N]} \| M_{\mathbf{p}_i}(t) \|_\infty \ge 10 \cR | \cS | e^T \theta(W) + \frac{\epsilon}{2 e^{L T}} \right) \\
& \hspace{2cm} \stackrel{(c)}{\le} | \widetilde{\Gamma}_\delta | | \cS |\mathrm{exp} \left( - \frac{ N \epsilon^2}{1000 \cR | \cS|^2 e^{(2L   + 1)T} } \right),
\end{align*}
where $(a)$ is due to \eqref{eq:deviation_gronwall_result}; $(b)$ follows from the choice of $\delta$, \eqref{eq:M_P_Q_bound} and Lemma \ref{lemma:covering_number}; finally, $(c)$ follows from a union bound and an application of Lemma \ref{lemma:main_concentration}. To conclude the proof, 
we can let $h(T, \epsilon)$ be the bound on $| \widetilde{\Gamma}_\delta|$ from Lemma \ref{lemma:covering_number}, multiplied by $| \cS |$.
\end{proof}

\section{Properties of $\Gamma(P,T)$ and $\Gamma_\infty$: Proofs of Lemmas \ref{lemma:Gamma_properties} and \ref{lemma:covering_number}}
\label{sec:gamma}

We begin by explicitly defining the set $\Gamma(P, T)$. For $\alpha, \beta \in \cS$ and $t \ge 0$, let $\boldsymbol{\rho}^{\alpha \beta}(t)$ be the $N \times N$ diagonal matrix with $i$th diagonal entry given by $\rho^{\alpha \beta}( \overline{y}_i(t))$. We also let $\mathrm{Diag}(\mathbf{r})$ be the diagonal matrix with $i$th diagonal entry equal to $r_i$. We now define 
\begin{equation}
\label{eq:gamma_p_t}
\Gamma(P, T) : = \{ W \} \cup \left \{ \mathrm{Diag}(r) \boldsymbol{\rho}^{\alpha \beta}(t) P : \alpha, \beta \in \cS; t \in [0,T] \right \}.
\end{equation}

\begin{proof}[Proof of Lemma \ref{lemma:Gamma_properties}]
We start by proving Item~\#\ref{item:Gamma_basic_properties}. Let $Q \in \Gamma(P, T)$. If $Q = W$, the claim is immediate. Else if $Q \neq W$, we can write $Q = \mathrm{Diag}(\mathbf{r}) \boldsymbol{\rho}^{\alpha \beta}(t) P$ for some $\alpha, \beta \in \cS$ and $t \in [0,T]$. In particular, we can bound the entries of $Q$ as $q_{ij} = r_i \rho^{\alpha \beta}(\overline{y}_i(t)) p_{ij} \le p_{ij}$, which follows since $r_i < 1$ and $\rho^{\alpha \beta} \le 1$. It immediately follows that $\theta(Q) \le \theta(P)$ and that $Q$ has a maximum column sum at most $\cR$. 

We now turn to the proof of Item~\#\ref{item:Gamma_lipschitz}. Notice that for any $i \in [N]$ we can write $\Phi_j^\alpha( \mathbf{Y}(t)) - \Phi_j^\alpha(\mathbf{y}(t)) = \sum_{\beta \in \cS} ( f_j^{\beta \alpha}(t) - f_j^{\alpha \beta}(t) )$, where
\begin{align*}
f_j^{\alpha \beta}(t) & : = r_j Y_j^\alpha(t) \rho^{\alpha \beta} ( \overline{Y}_j(t)) - r_j y_j^\alpha(t) \rho^{\alpha \beta}( \overline{y}_j(t)) \\
& = r_j Y_j^\alpha(t) \left( \rho^{\alpha \beta} ( \overline{Y}_j(t)) -  \rho^{\alpha \beta}( \overline{y}_j(t)) \right) + r_j\rho^{\alpha \beta}( \overline{y}_j(t)) ( Y_j^\alpha(t) - y_j^\alpha(t) ) \\
& \le L_\rho \| \overline{Y}_j(t) - \overline{y}_j(t) \|_\infty + r_j \rho^{\alpha \beta}( \overline{y}_j(t)) ( Y_j^\alpha(t) - y_j^\alpha(t)),
\end{align*}
where the last inequality follows since $\rho^{\alpha \beta}$ is $L_\rho$-Lipschitz by Assumption \ref{as:rho} and since $r_j <1$. Next, define $Q : = \mathrm{Diag}(\mathbf{r}) \boldsymbol{\rho}^{\alpha \beta}(t) P$. Taking a weighted sum with respect to $\mathbf{p}_i$, we obtain
$$
| f_{\mathbf{p}_i}^{\alpha \beta}(t)|  : = \left| \sum\limits_{j \in [N]} p_{ij} f_j^{\alpha \beta}( t) \right | \le L_\rho \sum\limits_{j \in [N]} p_{ij} \| \overline{Y}_j(t) - \overline{y}_j(t) \|_\infty +  \| Y_{\mathbf{q}_i}(t) - y_{\mathbf{q}_i}(t) \|_\infty,
$$
where $\mathbf{q}_i$ is the $i$th row vector of $Q$. Now taking an average over $i \in [N]$,
\begin{align}
\frac{1}{N} \sum\limits_{i \in [N]}  \left| f_{\mathbf{p}_i}^{\alpha \beta}(t)\right|  & \le \frac{L_\rho}{N} \sum\limits_{j \in [N]} \left( \sum\limits_{i \in [N]} p_{ij} \right) \| \overline{Y}_j(t) - \overline{y}_j(t) \|_\infty + A(Q, t) \nonumber \\
\label{eq:f_abs_bound}
& \hspace{-1cm} \le \frac{L_\rho \cR}{N} \sum\limits_{j \in [N]} \| \overline{Y}_j(t) - \overline{y}_j(t) \|_\infty + A(Q, t) 
 \le (L_\rho \cR + 1) \sup\limits_{Q \in \Gamma(P, T)} A(Q, t). 
\end{align}
Above, the second inequality uses the bound on the maximum column sum of $P$ and the third inequality follows since $W, Q \in \Gamma(P, T)$.
Finally, to bound the quantity of interest, we have
\begin{align}
\label{eq:Phi_f_bound}
\frac{1}{N} \sum\limits_{i \in [N]} \left | \Phi^\alpha_{\mathbf{p}_i} ( \mathbf{Y}(t) ) - \Phi^\alpha_{\mathbf{p}_i} ( \mathbf{y}(t) ) \right | & \le \frac{1}{N} \sum\limits_{i \in [N]} \sum\limits_{\alpha, \beta \in \cS} | f_{\mathbf{p}_i}^{\alpha \beta}(t) |.
\end{align}
The desired result follows from combining \eqref{eq:f_abs_bound} and \eqref{eq:Phi_f_bound}. 
\end{proof}

We now turn to the proof of Lemma \ref{lemma:covering_number}, which concerns $\Gamma_\infty$ (see Definition \ref{def:Gamma_infinity}).

\begin{proof}[Proof of Lemma \ref{lemma:covering_number}]
For a non-negative integer $m$ as well as vectors $\boldsymbol{\alpha}, \boldsymbol{\beta} \in \cS^m$ and $\mathbf{t} \in [0,T]^m$, define the matrices $P(m, \boldsymbol{\alpha}, \boldsymbol{\beta}, \mathbf{t})  : = \mathrm{Diag}(\mathbf{r})^m \left( \prod_{\ell = 1}^m \boldsymbol{\rho}^{\alpha_\ell \beta_\ell} (t_\ell) \right) W$
as well as
$Q(m, \boldsymbol{\alpha}, \boldsymbol{\beta}, \mathbf{t}) := \mathrm{Diag}(\mathbf{r})^m \left( \prod_{\ell = 1}^m \boldsymbol{\rho}^{\alpha_\ell \beta_\ell} (t_\ell) \right) \mathbf{11}^\top / N$.
It can be readily seen from the definition of $\Gamma(P, T)$ and Definition \ref{def:Gamma_infinity} that $\Gamma_\infty = \Gamma_\infty^1 \cup \Gamma_\infty^2$, where $\Gamma_\infty^1$ is the set of all matrices $P( m, \boldsymbol{\alpha}, \boldsymbol{\beta}, \mathbf{t})$ where $m \in \integers_{\ge 0}$, $\boldsymbol{\alpha}, \boldsymbol{\beta} \in \cS^m$, and $\mathbf{t} \in [0,T]^m$. The set $\Gamma_\infty^2$ is of the same form with $P$ replaced by $Q$. Our goal is to bound the covering number of $\Gamma_\infty^1$ and $\Gamma_\infty^2$. 

For a fixed $\eta > 0$ and a positive integer $M$, define $\mathbb{T}^\eta : = \{ k \eta : k \in \integers_{\ge 0} \}$ as well as the set $\widetilde{\Gamma} : = \{ \mathbf{0} \} \cup \{ P( m, \boldsymbol{\alpha}, \boldsymbol{\beta}, \mathbf{s} ) \in \Gamma_\infty^1 : m \in [M]; \mathbf{s} \in ( [0,T] \cap \mathbb{T}^\eta)^m \}$. We claim that for each $X \in \Gamma_\infty^1$, there exists $\revision{X'} \in \widetilde{\Gamma}$ such that $\| X - X' \|_1 /N \le \delta$,
where $\| \cdot \|_1$ is the sum of the absolute entries of the input matrix.
If $X = P(m, \boldsymbol{\alpha}, \boldsymbol{\beta}, \boldsymbol{t})$ with $m \ge M$, then we can bound the entries of $X$ by $X_{ij} \le r_{max}^m w_{ij} \le r_{max}^M w_{ij}$ for all $i,j$ since $r_{max} < 1$. Setting $X' = \mathbf{0}$, it follows that $\| X - X'\|_1 / N = \| X \|_1 / N \le r_{max}^M$. On the other hand, suppose that $m < M$. Setting $X' = P(m, \boldsymbol{\alpha}, \boldsymbol{\beta}, \mathbf{s} )$ where $\mathbf{s} \in ( [0,T] \cap \mathbb{T}^\eta)^m$ with $\| \mathbf{t} - \mathbf{s} \|_\infty \le \eta$, we can bound the entries of $X - X'$ by 
\begin{equation}
\label{eq:X_X'_difference}
| X_{ij} - X'_{ij} | = r_i^m \left| \prod\limits_{\ell = 1}^m \rho^{\alpha_\ell \beta_\ell}( \overline{y}_i(t_\ell) ) - \prod\limits_{\ell = 1}^m \rho^{\alpha_\ell \beta_\ell} ( \overline{y}_i(s_\ell) ) \right| w_{ij} \le ( {m} L_\rho \eta) w_{ij},
\end{equation}
where the final inequality 
uses the fact that the product of $m$ $L_\rho$-Lipschitz functions bounded by 1 in magnitude is $(mL_\rho)$-Lipschitz.
Since $m \le M$, \eqref{eq:X_X'_difference} implies
that $\| X - X' \|_1 / N \le M L_\rho \eta$. 

To conclude the proof, it remains to choose appropriate values of $M, \eta$. In particular, we require that $\max \{ r_{max}^M, M L_\rho \eta \} \le \delta$, so we may set $M = \lceil | ( \log \delta) / ( \log r_{max}) | \rceil$ and $\eta = \delta / (M L_\rho)$. Moreover, it is readily seen through simple counting arguments that $| \widetilde{\Gamma}| \le M ( | \cS |^2 T / \eta)^M + 1 \le 2 M ( | \cS |^2 T / \eta)^M \le (e |S|^2 T / \eta )^M$, where the final inequality uses $2M \le e^M$ for $M \ge 1$. Through identical arguments we can also bound the covering number of $\Gamma_\infty^2$, and the desired result follows. 
\end{proof}
\vspace{-0.3cm}
\section{Martingale tail inequalities{: Proof of Lemma \ref{lemma:main_concentration}}}
\label{sec:martingale_tail_inequalities}
To prove the lemma, we will primarily work with the process $V_\delta^\alpha(t) : = \sqrt{ \sum_{i = 1}^N M_{\mathbf{p}_i}^\alpha(t)^2/N + \delta^2}$, where $\delta > 0$ is a fixed constant and $\alpha \in \cS$. This can be related to the process of interest as follows: 
\begin{equation}
\label{eq:M_V_inequality}
\sum_{i \in [N]} \frac{ \| M_{\mathbf{p}_i}(t) \|_\infty }{N} \le \sum_{\substack{\alpha \in \cS , i \in [N]}} \frac{ | M_{\mathbf{p}_i}^\alpha(t) | }{N} \le \sum_{\alpha \in \cS } \sqrt{ \sum_{i \in [N]} \frac{ M_{\mathbf{p}_i}^\alpha(t)^2 }{N} } \le \sum_{\alpha \in \cS} V^\alpha_\delta(t).
\end{equation}
An important consequence of \eqref{eq:M_V_inequality} is that any upper tail bounds we derive for $V^\alpha_\delta(t)$ can be translated to the process of interest. The following lemma, which is the key supporting result for Lemma \ref{lemma:main_concentration}, establishes some useful properties of $V^\alpha_\delta(t)$. 

\begin{lemma}
\label{lemma:V_properties}
Let $P \in \mathbb{R}^{N \times N}$ be a non-negative, sub-stochastic matrix with column sums bounded by $\cR$ and $\theta(P) \le \theta(W)$. Suppose also that $\delta^2 \ge \max \{ \theta(W)^2, 4 (T+1) \cR / N \}$. Then the following hold for $\xi$ sufficiently small: 
\begin{enumerate}
\item \label{item:V_increment_bound}
$\E[ V_\delta^\alpha(t + \xi) - V_\delta^\alpha(t) \vert \cF_t ] \le  \xi \theta(W) + o( \xi)$;
\item \label{item:V_absolute_jump}
It holds almost surely that all jumps of $V_\delta^\alpha(t)$ are at most $\frac{(T+1)\cR}{2 \delta N}$. 
\item \label{item:V_quadratic_variation}
$\E [ (V_\delta^\alpha(t + \xi) - V_\delta^\alpha(t))^2 \vert \cF_t ] \le \frac{8 \cR \xi}{N} + o(\xi)$. 
\end{enumerate}
\end{lemma}
The proof follows from standard but tedious calculations; we defer the details to Appendix \ref{sec:martingale_properties}. The next ingredient of the proof is a version of Freedman's inequality; we state it below for completeness. 

\begin{lemma}
\label{lemma:freedman}
Suppose that $\{X_t \}_{t \ge 0}$ is a continuous-time supermartingale adapted to the filtration $\{ \cF_t \}_{t \ge 0}$, and suppose further that $X_t$ has jumps that are bounded by $C$ in magnitude, almost surely. Then for any $x \ge 0$ and $\sigma > 0$, 
\[
\p ( X_t - X_0 \ge x \text{ and } \langle X \rangle_t \le \sigma^2 \text{ for some $t \ge 0$} ) \le \mathrm{exp} \left( - \frac{ x^2}{ 2( \sigma^2 + Cx/3)} \right),
\]
where $\langle X \rangle_t : = \lim_{\xi \to 0} \sum_{k = 0}^{\lfloor t / \xi \rfloor}  \E [ ( X_{(k + 1) \xi} - X_{k\xi})^2 \vert \cF_{k\xi} ]$ is the quadratic variation. 
\end{lemma}
Lemma \ref{lemma:freedman} was previously proved for continuous-time martingales by Shorack and Wellner, but their proof readily extends to the case of supermartingales as well; we defer the interested reader to \cite[Appendix B]{shorack_wellner} for details. 
\vspace{-0.2cm}
\begin{proof}[Proof of Lemma \ref{lemma:main_concentration}]
Consider the process $\widetilde{V}_\delta^\alpha(t) : = V_\delta^\alpha(t) - t \theta(W)$, which is a supermartingale by Item \#\ref{item:V_increment_bound} of Lemma \ref{lemma:V_properties}. The magnitude of the jumps of $\widetilde{V}_\delta^\alpha(t)$ is the same as that of $V_\delta^\alpha(t)$, and we further have that, for $\xi$ sufficiently small, 
\begin{multline}
\label{eq:V_tilde_quadratic_variation}
\E \left[ \left.  \left( \widetilde{V}_\delta^\alpha(t + \xi) - \widetilde{V}_\delta^\alpha(t) \right)^2 \right \vert \cF_t \right] = \E \left[ \left. \left( V_\delta^\alpha(t + \xi) - V_\delta^\alpha(t) -  \xi \theta(W) \right)^2 \right \vert \cF_t \right] \\
\le 2 \E \left[ \left. \left( V_\delta^\alpha(t + \xi) - V_\delta^\alpha(t) \right)^2 \right \vert \cF_t \right] + O ( \xi^2 ) \le \frac{16 \cR \xi }{N} + o(\xi),
\end{multline}
where the first inequality uses the relation $(a + b)^2 \le 2 a^2 + 2b^2$, and the second inequality follows from Item \#\ref{item:V_quadratic_variation} of Lemma \ref{lemma:V_properties}. An immediate consequence of the bound in \eqref{eq:V_tilde_quadratic_variation} is that $\langle \widetilde{V}_\delta^\alpha \rangle_t \le 16 t \cR / N$, which is, almost surely, at most $16 T \cR / N$ for $t \in [0,T]$. An application of Lemma \ref{lemma:freedman} to the process $\widetilde{V}_\delta^\alpha(t)$ yields
\begin{equation}
\label{eq:V_tilde_probability_bound}
\p \left( \sup_{0 \le t \le T} \{ \widetilde{V}_\delta^\alpha(t) - \widetilde{V}_\delta^\alpha(0) \} \ge \delta \right) \le \mathrm{exp} \left( - \frac{N \delta^2 / 2}{16 T \cR + (T+1) \cR / 6} \right) \hspace{-0.08cm} \le  \mathrm{exp} \left( - \frac{ N \delta^2}{40\cR e^T} \right). 
\end{equation}
Next, note that since $\widetilde{V}_\delta^\alpha(0) = \delta$, it holds for $t \in [0,T]$ that $V_\delta^\alpha(t) - T \theta(W) - \delta \le \widetilde{V}_\delta^\alpha(t) - \widetilde{V}_\delta^\alpha(0)$, hence $\sup_{0 \le t \le T} V_\delta^\alpha(t) \ge T \theta(W) + 2 \delta \Rightarrow \sup_{0 \le t \le T} \{ \widetilde{V}_\delta^\alpha(t) - \widetilde{V}_\delta^\alpha(0) \} \ge \delta$. In particular, the probability bound established in \eqref{eq:V_tilde_probability_bound} also holds for the event $\{\sup_{0 \le t \le T} V_\delta^\alpha(t) \ge T \theta(W) + 2 \delta \}$. 
It remains to choose $\delta$. In light of the restrictions on $\delta$ in Lemma \ref{lemma:V_properties}, and since $\theta(W)^2 \ge 1/N$ (see Remark \ref{remark:spectral_density_homogenous}), a valid choice is $\delta : = 4(1 + T )\cR \theta(W) + \frac{\epsilon}{2 | \cS |}$ (where $\epsilon > 0$ is arbitrary). Putting everything together shows that
\[
\p \left( \sup_{0 \le t \le T} V_\delta^\alpha(t) \ge 10T \cR ( 1 + T  ) \theta(W) + \frac{\epsilon}{| \cS |} \right) \le \mathrm{exp} \left( - \frac{N \epsilon^2}{160  \cR | \cS |^2 e^T } \right). 
\] 
Above, we have used \eqref{eq:V_tilde_probability_bound} and $\delta^2 \ge \frac{\epsilon^2}{4 | \cS|^2}$. Taking a union bound over $\alpha \in \cS$, and using the inequality $1 + x \le e^x$ to simplify terms in the event of interest, we have that 
\[
\p \left( \sup_{0 \le t \le T} \sum_{\alpha \in \cS} V_\delta^\alpha(t) \ge 10 T \cR | \cS | e^{T} \theta(W) + \epsilon \right) \le | \cS | \mathrm{exp} \left( - \frac{ N \epsilon^2}{160  \cR | \cS |^2 e^T } \right).
\]
Finally, the desired result follows readily from \eqref{eq:M_V_inequality}. 
\end{proof}
\vspace{-0.3cm}
\section{Conclusion}
We established a generic theory that reveals when stochastic population processes are characterized by their mean-field approximation. Based on whether agent interactions are spectrally or locally dense, the CMFA or the NIMFA may be most accurate. Our technical results establish exponential concentration inequalities for non-Markov processes, which may be of independent interest. We also illustrated through simulations that using the CMFA instead of the NIMFA can lead to significant errors in understanding stochastic population processes. 
In future work, we will derive richer characterizations of the stochastic population process {\it beyond} finite time horizons (e.g., metastability) and tailor our theory to further applications in game theory and epidemiology.

\vspace{-0.6cm} 

\bibliographystyle{siamplain}
\bibliography{mf}

\appendix

\section{Proofs of Lemmas \ref{lemma:Mav_concentration} and \ref{lemma:V_properties}}
\label{sec:martingale_properties}
We start by proving some useful intermediate results about the $M_{\mathbf{p}_i}(t)$'s. 
\begin{lemma}
\label{lemma:M_properties}
Let $P \in \mathbb{R}^{N \times N}$ be non-negative and sub-stochastic with maximum column sum at most $\cR$. For all $t \ge 0$ and all $\xi$ sufficiently small, 
\begin{enumerate}
\item \label{item:M_as_bound}
Almost surely, $| M_{\mathbf{p}_i}^\alpha(t) | \le 1 + t$ for all $i \in [N]$ and $\alpha \in \cS$; 
\item \label{item:M_squared_difference}
$\E [ \frac{1}{N} \sum_{i \in [N]} ( M_{\mathbf{p}_i}^\alpha(t + \xi)^2 - M_{\mathbf{p}_i}^\alpha(t)^2 )  \vert \cF_t ] \le 2 \xi \theta(P)^2 + o ( \xi);$
\item \label{item:M_squared_squared}
$\E [ (\frac{1}{N} \sum_{i \in [N]} ( M_{\mathbf{p}_i}^\alpha(t + \xi)^2 - M_{\mathbf{p}_i}^\alpha(t)^2 ) )^2 \vert \cF_t ] \le \frac{8 \cR \xi}{N} ( \theta(P)^2 + \frac{1}{N} \sum_{i \in [N]} M_{\mathbf{p}_i}^\alpha(t)^2 )$.
\end{enumerate}
\end{lemma}
\begin{proof}
We start by proving Item \#\ref{item:M_as_bound}. To this end, we can bound $|M_{\mathbf{p}_i}^\alpha(t) | \le |Y_{\mathbf{p}_i}^\alpha(t) - Y_{\mathbf{p}_i}^\alpha(t)| + | \int_0^t \Phi_{\mathbf{p}_i}^\alpha( \mathbf{Y}(s)) ds |$. The first term is at most 1 since $Y_j^\alpha(t) \in \{0,1 \}$ and $P$ is a sub-stochastic matrix; the second term is at most $t$ since $| \Phi_{\mathbf{p}_i}^\alpha ( z) | \le 1$ for all $z \in \Delta( \cS)^N$ (see \eqref{eq:mean_field}). The desired claim immediately follows. 
Next, we prove Item \#\ref{item:M_squared_difference}. By the martingale property of $M_{\mathbf{p}_i}(t)$,
\begin{align}
\E [ M_{\mathbf{p}_i}^\alpha(t + \xi)^2 - M_{\mathbf{p}_i}^\alpha(t)^2 \vert \cF_t ] & = \E [ (M_{\mathbf{p}_i}^\alpha(t + \xi) - M_{\mathbf{p}_i}^\alpha(t) )^2 \vert \cF_t ] \nonumber \\
\label{eq:M_squared_difference}
& \hspace{-3cm} \le 2 \E [ ( Y_{\mathbf{p}_i}^\alpha(t + \xi) - Y_{\mathbf{p}_i}^\alpha(t) )^2 \vert \cF_t ] + 2 \E \left[ \left( \int_t^{t + \xi} \xi \Phi_{\mathbf{p}_i}^\alpha( \mathbf{Y}(s) ) ds \right)^2 \vert \cF_t \right]
\end{align}
Since $| \Phi_{\mathbf{p}_i}^\alpha( \mathbf{Y}(s) ) | \le 1$, the second term above is $O( \xi^2)$. To bound the first term above, notice that $|Y_j^\alpha(t + \xi) - Y_j^\alpha(t)|$ is at most the number of times agent $j$'s clock rings in the interval $[t, t + \xi]$, which is a $\mathrm{Poi}( \xi r_j)$ random variable. As agent clock rings are independent and since we assume $r_j < 1$ for all $j \in [N]$, $| Y_{\mathbf{p}_i}^\alpha (t + \xi) - Y_{\mathbf{p}_i}^\alpha(t) |$ can be stochastically bounded by $\sum_{j \in [N]} p_{ij} X_j$, where the $X_j$'s are independent $\mathrm{Poi}( \xi)$ random variables. Hence 
\begin{multline*}
\E [ ( Y_{\mathbf{p}_i}^\alpha( t + \xi) - Y_{\mathbf{p}_i}^\alpha(t) )^2 \vert \cF_t ]  \le \E \left[ \left( \sum_{j \in [N]} p_{ij} X_j \right)^2 \right] \\
 = \sum_{j \in [N]} p_{ij}^2 \E [ X_j^2 ] + \sum_{j , k \in [N]: j \neq k } p_{ij} p_{ik} \E [ X_j] \E [ X_k] = \xi \sum_{j \in [N]} p_{ij}^2 + O ( \xi^2 ). 
\end{multline*}
The display above combined with \eqref{eq:M_squared_difference} proves Item \#\ref{item:M_squared_difference}. 
Finally, we prove Item \#\ref{item:M_squared_squared}. For $i \in [N]$, define $U_i^\alpha(t) : = M_{\mathbf{p}_i}^\alpha(t + \xi) - M_{\mathbf{p}_i}^\alpha(t) = Y_{\mathbf{p}_i}^\alpha(t + \xi) - Y_{\mathbf{p}_i}^\alpha(t) - \int_t^{t + \xi} \Phi_{\mathbf{p}_i}^\alpha( \mathbf{Y}(s)) ds$ as well as $U_{\mathbf{p}_i}^\alpha(t) : = \sum_{j \in [N]} p_{ij} U_j^\alpha(t)$. Note that we have the representation $M_{\mathbf{p}_i}^\alpha(t + \xi)^2 - M_{\mathbf{p}_i}^\alpha(t)^2 = U_{\mathbf{p}_i}^\alpha(t)^2 + 2 U_{\mathbf{p}_i}^\alpha(t) M_{\mathbf{p}_i}^\alpha(t)$. Using this representation as well as the inequality $( x + y)^2 \le 2 x^2 + 2 y^2$, we can upper bound the conditional expectation in Item \#\ref{item:M_squared_squared} by 
\begin{equation}
\label{eq:U_UM_decomposition}
\frac{2}{N^2} \E \left[ \left. \left( \sum_{i \in [N]} U_{\mathbf{p}_i}^\alpha(t)^2 \right)^2 \right \vert \cF_t \right] + \frac{8}{N^2} \E \left[ \left. \left( \sum_{i \in [N]} U_{\mathbf{p}_i}^\alpha(t) M_{\mathbf{p}_i}^\alpha(t) \right)^2 \right \vert \cF_t \right]
\end{equation}
We start by bounding the first term in \eqref{eq:U_UM_decomposition}, which essentially amounts to characterizing terms of the form $U_{\mathbf{p}_i}^\alpha(t)^2 U_{\mathbf{p}_j}^\alpha(t)^2$.
By the definition of $U_{\mathbf{p}_i}^\alpha(t)$,
we have that
\begin{align}
U_{\mathbf{p}_i}^\alpha(t)^2 & \le 2 \left( Y_{\mathbf{p}_i}^\alpha( t + \xi) - Y_{\mathbf{p}_i}^\alpha(t) \right)^2 + 2 \left( \int_t^{t + \xi} \Phi_{\mathbf{p}_i}^\alpha( \mathbf{Y}(s) ) ds \right)^2 \nonumber \\
\label{eq:U_squared}
&  \le 2 ( Y_{\mathbf{p}_i}^\alpha(t + \xi) - Y_{\mathbf{p}_i}^\alpha(t) )^2 + 2 \xi^2.
\end{align}
Above, the first inequality uses $(a + b)^2 \le 2 a^2 + 2b^2$, and the second inequality uses that $| \Phi_{\mathbf{p}_i}^\alpha(z) | \le 1$ for all $z \in \Delta(\cS)^N$. Now, to study the final expression in \eqref{eq:U_squared}, let $\{ X_k \}_{k \in [N]}$ be a collection of i.i.d. $\mathrm{Poi}(\xi)$ random variables. Using the same arguments in the proof of Item \#\ref{item:M_squared_difference}, it is readily seen that $| Y_{\mathbf{p}_i}^\alpha(t + \xi) - Y_{\mathbf{p}_i}^\alpha(t) |$ and $| Y_{\mathbf{p}_j}^\alpha(t + \xi) - Y_{\mathbf{p}_j}^\alpha(t) |$ are stochastically dominated by $\sum_{k \in [N]} p_{ik} X_k$ and $\sum_{k \in [N]} p_{jk} X_k$, respectively. Hence
\begin{align*}
\E [ U_{\mathbf{p}_i}^\alpha(t)^2 U_{\mathbf{p}_j}^\alpha(t)^2 \vert \cF_t ] & \le 4 \E \left[ \left( \sum_{k \in [N]} p_{ik} X_k \right)^2 \left( \sum_{k \in [N]} p_{jk} X_k \right)^2 \right] + o ( \xi) \\
& \hspace{-2cm} = 4  \sum_{a,b,c,d \in [N]} p_{ia} p_{ib} p_{jc} p_{jd} \E [ X_a X_b X_c X_d ] + o ( \xi)  = 4 \xi \sum_{k \in [N]} p_{ik}^2 p_{jk}^2 + o(\xi).
\end{align*}
In the display above, the first inequality is due to \eqref{eq:U_squared} and the stochastic dominance arguments, the equality on the second line is due to expanding the squares, and the final equality follows since the probability that both $X_a$ and $X_b$ are positive for distinct $a,b$ is $O(\xi^2)$. Putting everything together, we have that
\begin{align}
\E \left[ \left. \left( \sum_{i \in [N]} U_{\mathbf{p}_i}^\alpha(t)^2 \right)^2  \right \vert \cF_t \right] & = \sum_{i,j \in [N]} \E [ U_{\mathbf{p}_i}^\alpha(t)^2 U_{\mathbf{p}_j}^\alpha(t)^2 \vert \cF_t ] \le 4 \xi \sum_{i,j,k \in [N]} p_{ik}^2 p_{jk}^2 + o(\xi) \nonumber \\
\label{eq:U_squared_squared_bound}
& \le 4 \xi \sum_{i,k \in [N]} p_{ik}^2 \left( \sum_{j \in [N]} p_{jk} \right) + o( \xi) \le 4 \cR \xi \| P \|_F^2 + o ( \xi). 
\end{align}
In the final inequality above, we have used that the maximum column sum in $P$ is at most $\cR$. 
We now bound the second term in \eqref{eq:U_UM_decomposition}. To this end, we will use the representation $\sum_{i \in [N]} U_{\mathbf{p}_i}^\alpha(t) M_{\mathbf{p}_i}^\alpha(t) = \sum_{j \in [N]} B_j U_j^\alpha(t)$, where $B_j : = \sum_{i \in [N]} M_{\mathbf{p}_i}^\alpha(t) p_{ij}$. Since $B_j$ is $\cF_t$-measurable for all $j \in [N]$,
it follows that $\E [ ( \sum_{i \in [N]} U_{\mathbf{p}_i}^\alpha(t) M_{\mathbf{p}_i}^\alpha(t) )^2 \vert \cF_t]$ $= \sum_{j,k} B_j B_k \E [ U_j^\alpha(t) U_k^\alpha(t) \vert \cF_t ]$. To bound the conditional expectation in the summation, notice that $U_j^\alpha(t) \le | Y_j^\alpha(t + \xi) - Y_j^\alpha(t) | + | \int_t^{t + \xi} \Phi_j^\alpha( \mathbf{Y}(s) ) ds | \le | Y_j^\alpha(t + \xi) - Y_j^\alpha(t) | + \xi$. Hence
\begin{multline}
\label{eq:Uj_Uk}
\E [ U_j^\alpha(t) U_k^\alpha(t) \vert \cF_t ]  \le \E [ | Y_j^\alpha(t + \xi) - Y_j^\alpha(t) | | Y_k^\alpha(t + \xi) - Y_k^\alpha(t) | \vert \cF_t ] \\
+ \xi ( \E [ | Y_j^\alpha(t + \xi) - Y_j^\alpha(t) | \vert \cF_t ] + \E [ | Y_k^\alpha(t + \xi) - Y_k^\alpha(t) | \vert \cF_t ] ) + o( \xi).
\end{multline}
Again, using the stochastic dominance of $\{ | Y_j^\alpha(t + \xi) - Y_j^\alpha(t) | \}_{j \in [N]}$ by $\{ X_j \}_{j \in [N]}$, a sequence of i.i.d. $\mathrm{Poi}(\xi)$ random variables, \eqref{eq:Uj_Uk} implies that $\E[ U_j^\alpha(t) U_k^\alpha(t) \vert \cF_t ] = O(\xi^2)$ for $j \neq k$, and $\E [ U_j^\alpha(t)^2 \vert \cF_t ] \le \xi + O(\xi^2)$. We can then bound
\begin{align}
\sum_{j,k \in [N]} B_j B_k \E [ U_j^\alpha(t) U_k^\alpha(t) \vert \cF_t ] & \le \xi \sum_{j \in [N]} B_j^2 + o ( \xi) = \xi \cR^2 \sum_{j \in [N]} \left( \sum_{i \in [N]} M_{\mathbf{p}_i}^\alpha(t) \frac{p_{ij}}{\cR} \right)^2 \nonumber \\
\label{eq:UM_B_bound}
& \le \xi \cR \sum_{i,j \in [N]} M_{\mathbf{p}_i}^\alpha(t)^2 p_{ij} \le \xi \cR \sum_{i \in [N]} M_{\mathbf{p}_i}^\alpha(t)^2. 
\end{align}
Above, the second inequality is due to Jensen's inequality, since all column sums of $P$ are at most $\cR$.
Together, \eqref{eq:U_UM_decomposition}, \eqref{eq:U_squared_squared_bound} and \eqref{eq:UM_B_bound} imply the desired result. 
\end{proof}
We are now ready to prove the main results of this section. 
\begin{proof}[Proof of Lemma \ref{lemma:Mav_concentration}]
Set $P : = \mathbf{11}^\top/ N$, so that $M_{\mathbf{p}_i}(t) = M_{av}(t)$ and $\theta(P)^2 = 1/N$. It is readily seen from the definition of $M_{av}(t)$ that its jumps
are almost surely at most $1/N$, and Item \#\ref{item:M_squared_difference} of Lemma \ref{lemma:M_properties} implies that $\langle M_{av}^\alpha \rangle_t \le 2T/N$ for $t \in [0,T]$. The desired result now follows from applying Lemma \ref{lemma:freedman} to the supermartingales $M_{av}^\alpha(t)$ and $- M_{av}^\alpha(t)$, as well as taking a union bound over $\alpha \in \cS$. 
\end{proof}
\begin{proof}[Proof of Lemma \ref{lemma:V_properties}]
We start with the useful chain of inequalities
\begin{equation}
\label{eq:V_increment_bound_1}
V^\alpha_\delta( t+ \xi) - V^\alpha_\delta(t) \le \sum_{i \in [N]} \frac{ M_{\mathbf{p}_i}^\alpha( t + \xi)^2 - M_{\mathbf{p}_i}^\alpha(t)^2 }{2N V^\alpha_\delta(t) } \le \sum_{i \in [N]} \frac{ M_{\mathbf{p}_i}^\alpha(t + \xi)^2 - M_{\mathbf{p}_i}^\alpha(t)^2 }{2N \theta(W)}.
\end{equation}
Above, the first inequality uses that $\sqrt{x} - \sqrt{y} \le (x - y) / (2 \sqrt{y} )$ for $x, y > 0$ (which follows from the concavity of the square root function), and the second inequality uses $V^\alpha_\delta(t) \ge \delta \ge \theta(W)$. Item \#\ref{item:V_increment_bound} now follows immediately by taking an expectation on both sides with respect to $\cF_t$ and invoking Item \#\ref{item:M_squared_difference} of Lemma \ref{lemma:M_properties}. 
We now turn to the proof of Item \#\ref{item:V_absolute_jump}.
Noting that the function $x \mapsto \sqrt{ x + \delta^2}$ for $x \ge 0$ is $1/(2\delta)$-Lipschitz, we have the following bound for $t , t + \xi \in [0,T]$:
\begin{align}
| V^\alpha_\delta(t + \xi) - V^\alpha_\delta(t) | & \le \frac{1}{2 \delta N} \sum_{i \in [N]} \left| M_{\mathbf{p}_i}^\alpha(t + \xi)^2 - M_{\mathbf{p}_i}^\alpha(t)^2 \right| \nonumber \\
& \hspace{-2cm} = \frac{1}{2 \delta N} \sum_{i \in [N]} \left | Y_{\mathbf{p}_i}^\alpha(t + \xi) - Y_{\mathbf{p}_i}^\alpha(t) - \int_t^{t + \xi} \Phi_{\mathbf{p}_i}^\alpha( \mathbf{Y}(s)) ds \right | | M_{\mathbf{p}_i}^\alpha(t + \xi) + M_{\mathbf{p}_i}^\alpha(t) | \nonumber \\
\label{eq:V_increment_bound_2}
& \hspace{-2cm} \le \frac{T+ 1}{\delta N} \left( \xi + \sum_{i \in [N]} | Y_{\mathbf{p}_i}^\alpha(t + \xi) - Y_{\mathbf{p}_i}^\alpha(t) | \right),
\end{align}
where, in the final inequality, we have used Item \#\ref{item:M_as_bound} of Lemma \ref{lemma:M_properties} to bound $| M_{\mathbf{p}_i}^\alpha(t + \xi) + M_{\mathbf{p}_i}^\alpha(t) |$ and also used that $| \Phi_{\mathbf{p}_i}^\alpha(z) | \le 1$ for all $z \in \Delta( \cS)^N$ to bound absolute value of the integral by $\xi$. 
Next, suppose that $t \in [0,T]$ is a point of discontinuity for $V^\alpha_\delta(t)$; then the size of the jump at $t$ is given by $\lim_{\xi \to 0} | V^\alpha_\delta( t) - V^\alpha_\delta(t - \xi) |$. When such a jump occurs, only one agent changes their state almost surely. If $j$ is the agent which changes their state at time $t$, we have from \eqref{eq:V_increment_bound_2} that 
$\lim_{\xi \to 0} | V^\alpha_\delta( t) - V^\alpha_\delta(t - \xi) |  \le  \sum_{i \in [N]} \frac{ (T+1) p_{ij}}{2 \delta N} \le \frac{ (T+1) \cR}{2 \delta N}$.
As this bound holds uniformly for all updating agents, Item \#\ref{item:V_absolute_jump} follows. 
Finally, we prove Item \#\ref{item:V_quadratic_variation}. By the concavity of the square root function, we have for $x, y > 0$ that $|\sqrt{x} - \sqrt{y} | \le |x - y| / (2 \min \{ \sqrt{x}, \sqrt{y} \})$; hence
\begin{equation}
\label{eq:V_increment_bound_3}
|V^\alpha_\delta(t + \xi) - V^\alpha_\delta(t) | \le \frac{ \left| \frac{1}{N} \sum_{i \in [N]} M_{\mathbf{p}_i}^\alpha(t + \xi)^2 - M_{\mathbf{p}_i}^\alpha(t)^2 \right| }{2 \min \{ V^\alpha_\delta(t), V^\alpha_\delta(t + \xi) \} }.
\end{equation}
We next work on simplifying the bound above. If no agent updates their state in the interval $[t, t + \xi]$, then $Y_{\mathbf{p}_i}^\alpha(t + \xi) = Y_{\mathbf{p}_i}^\alpha(t)$ for all $i \in [N]$, hence by \eqref{eq:V_increment_bound_2}, $V_\delta^\alpha(t + \xi) \ge V_\delta^\alpha(t) - (T+1) \xi / ( \delta N)$. If $\xi$ is sufficiently small so that $(T + 1)\xi / (\delta N) \le \delta / 2$, then it further holds that $V_\delta^\alpha(t + \xi) \ge V_\delta^\alpha(t) / 2$, since $V_\delta^\alpha(t) \ge \delta$. On the other hand, if only a single agent updates their state in the interval $[t, t + \xi]$, both \eqref{eq:V_increment_bound_2} and Item \#\ref{item:V_absolute_jump} imply that $V_\delta^\alpha(t + \xi) \ge V_\delta^\alpha(t) - (T+1) \xi / (\delta N) - (T+1) \cR / (2 \delta N)$. If $\xi$ is sufficiently small so that $(T+1) \xi / (\delta N) \le \delta / 4$ and $\delta^2 \ge 4 (T+1) \cR / N$, then $V_\delta^\alpha(t + \xi) \ge V_\delta^\alpha(t) - \delta / 2 \ge V_\delta^\alpha(t) / 2$. Now, since the probability that more than one agent updates in $[t, t + \xi]$ is $O( \xi^2)$, it holds with probability $1 - O(\xi^2)$ that $V_\delta^\alpha(t + \xi) \ge V_\delta^\alpha(t)/2$. Along with the boundedness of $V_\delta^\alpha(t)$ and \eqref{eq:V_increment_bound_3}, this shows
\begin{multline*}
\E [ ( V^\alpha_\delta(t + \xi) - V^\alpha_\delta(t) )^2 \vert \cF_t ] \le \frac{ \E [ \left. ( \frac{1}{N} \sum_{i \in [N]} M_{\mathbf{p}_i}^\alpha(t + \xi)^2 - M_{\mathbf{p}_i}^\alpha(t)^2 )^2 \right \vert \cF_t ] }{V^\alpha_\delta(t)^2 } + o ( \xi) \\
\le \frac{8 \cR \xi}{N} \left( \frac{ \theta(P)^2 + \frac{1}{N} \sum_{i \in [N]} M_{\mathbf{p}_i}^\alpha(t)^2 }{ V_\delta^\alpha(t)^2 } \right) + o(\xi) \le \frac{ 8 \cR \xi}{N} + o ( \xi). 
\end{multline*}
Above, the first inequality on the second line is due to Item \#\ref{item:M_squared_squared} of Lemma \ref{lemma:M_properties}, and the final inequality uses that $\delta \ge \theta(W) \ge \theta(P)$.
\end{proof}

\end{document}

%% file: ex_shared.tex
\usepackage{lipsum}
\usepackage{amsfonts}
\usepackage{graphicx}
\usepackage{epstopdf}
\usepackage{algorithmic}
\usepackage{subcaption}
\ifpdf
  \DeclareGraphicsExtensions{.eps,.pdf,.png,.jpg}
\else
  \DeclareGraphicsExtensions{.eps}
\fi

\newcommand{\reals}{\mathbb{R}}

\newcommand{\integers}{\mathbb{Z}}

\newcommand{\norm}[1]{\left\| #1 \right\|}

\newcommand{\cut}[1]{{}}

\newcommand{\E}{\mathbb{E}}
\newcommand{\p}{\mathbb{P}}

\def\cA{\mathcal{A}}  
  \def\cF{\mathcal{F}}

  \def\cR{\mathcal{R}}
\def\cS{\mathcal{S}}  \def\cU{\mathcal{U}}
\def\cV{\mathcal{V}}  
 
\newsiamremark{remark}{Remark}
\newsiamremark{hypothesis}{Hypothesis}
\crefname{hypothesis}{Hypothesis}{Hypotheses}
\newsiamthm{claim}{Claim}
\newsiamthm{question}{Question}
\newsiamthm{assumption}{Assumption}

\headers{Mean-field approximations for stochastic population processes}{A. Sridhar and S. Kar}

\title{Mean-field Approximations for Stochastic Population Processes with Heterogeneous Interactions
}

\author{Anirudh Sridhar\thanks{Department of Electrical and Computer Engineering, Princeton University, Princeton, NJ 
  (\email{anirudhs@princeton.edu}).}
\and Soummya Kar\thanks{Department of Electrical and Computer Engineering, Carnegie Mellon University, Pittsburgh, PA 
  (\email{soummyak@andrew.cmu.edu}).}
}

\usepackage{amsopn}
